\documentclass[12pt]{article}
\usepackage{amsmath}
\usepackage{amssymb}
\usepackage{geometry}
\usepackage{algorithm}
\usepackage{algorithmic}
\numberwithin{algorithm}{section}
\usepackage{amsthm}
\usepackage{graphicx}
\usepackage{multirow} 
\usepackage{url}
\usepackage[colorlinks=true,citecolor=blue,urlcolor=blue,draft=false]{hyperref}

\newtheorem{theorem}{Theorem}[section]
\newtheorem{lemma}[theorem]{Lemma}
\newtheorem{proposition}[theorem]{Proposition}
\newtheorem{corollary}[theorem]{Corollary}
\theoremstyle{remark}
\newtheorem{remark}[theorem]{Remark}

\theoremstyle{definition}
\newtheorem{definition}[theorem]{Definition}

\newcommand{\Diag}{\operatorname{Diag}}
\renewcommand{\vec}[1]{\boldsymbol{#1}}
\renewcommand{\th}{{}^{\mathrm{th}}}
\newcommand{\st}{\operatorname{s.t.}}
\newcommand{\normtr}{|\!|\!|}
\newcommand{\indic}{\mathbf{1}}

\newcommand{\Rplusstar}{\mathbb{R}_{+}^{*}}
\newcommand{\Splusplus}{\mathbb{S}_m^{++}}
 
\newcommand{\NN}{\mathbb{N}}
\newcommand{\RR}{\mathbb{R}} 
\newcommand{\EE}{\mathbb{E}}

\newcommand{\trace}{\operatorname{trace}}
\newcommand{\rank}{\operatorname{rank}}
\newcommand{\card}{\operatorname{card}}

\title{Approximation of a Maximum-Submodular-Coverage problem involving spectral functions, with application to Experimental Design}
\author{Guillaume Sagnol\footnote{
Parts of this work were done when the author was with INRIA Saclay \^Ile de France \& CMAP, \'Ecole Polytechnique, being supported
by Orange Labs through the research contract CRE EB 257676 with INRIA.
}\\
{\small Zuse Institut Berlin (ZIB), Takustr. 7, 14195 Berlin, Germany}\\
{\small \url{sagnol@zib.de}}
}
\date{}

\begin{document}

\maketitle

\begin{abstract}
We study a family of combinatorial optimization problems
defined by a parameter $p\in[0,1]$, which involves spectral
functions applied to positive semidefinite matrices, and has
some application in the theory of optimal experimental design.
This family of problems tends to a generalization of the classical
maximum coverage problem as $p$ goes to $0$, and to a trivial instance
of the knapsack problem as $p$ goes to $1$.

In this article, we establish a matrix inequality which shows that the objective function is submodular for all $p\in[0,1]$, from which it follows
that the greedy approach, which has often been used for this problem, always gives a design within $1-1/e$ of the optimum. 
We next study the design found by rounding the solution of the continuous relaxed problem, an approach which has been applied by several authors.
We prove an inequality which generalizes a classical result from the theory
of optimal designs, and allows us to give a rounding procedure with an approximation
factor which tends to $1$ as $p$ goes to $1$.
\end{abstract}


\section{Introduction}
This work is motivated by a generalization of the classical maximum coverage problem which arises
in the study of optimal experimental designs. This problem may be formally defined as follows:
given $s$ positive semidefinite matrices $M_1, \ldots, M_s$ of the same size
and an integer $N<s$, solve:
\begin{align}
 \max_{I \subset [s]} &\quad \rank \Big( \sum_{i \in I } M_i \Big) \label{maxrank_intro} \tag{$P_0$}\\
 \operatorname{s.t.} &\quad \card(I) \leq N, \nonumber
\end{align}
where we use the standard notation $[s]:=\{1,\ldots,s\}$ and $\card(S)$ denotes the cardinality of $S$.
When each $M_i$ is diagonal, it is easy to see that Problem~\eqref{maxrank_intro} is
equivalent to a max-coverage instance, by defining the sets $S_i=\{k: (M_i)_{k,k}>0\}$,
so that the rank in the objective of Problem~\eqref{maxrank_intro} is equal to
$\card \big(\cup_{i \in I} S_i \big)$.

A more general class of problems arising in the study of optimal experimental designs is obtained
by considering a \emph{deformation} of the rank which is defined through a spectral function.
Given $p\in[0,1]$, solve:
\begin{align}
 \max_{\vec{n}\in\NN^s} &\quad \varphi_p \left( \vec{n} \right) \label{Pp_intro} \tag{$P_p$}\\
 \operatorname{s.t.} &\quad \sum_{i\in[s]} n_i \leq N, \nonumber
\end{align}
where $\varphi_p(\vec{n})$ is the sum of the eigenvalues of $\sum_{i\in[s]} n_i M_i$ raised
to the exponent $p$:
if the eigenvalues of the positive semidefinite matrix $\sum_{i\in[s]} n_i M_i$
are $\lambda_1,\ldots,\lambda_m$ (counted with multiplicities), $\varphi_p(\vec{n})$ is defined by
$$\varphi_p(\vec{n}) = \trace \Big(\sum_{i\in[s]} n_i M_i \Big)^p = \sum_{k=1}^m  \lambda_k^p.$$
We shall see that Problem~\eqref{maxrank_intro} is the limit of Problem~\eqref{Pp_intro}
as $p \to 0^+$ indeed. On the other hand, the limit of Problem~\eqref{Pp_intro} as
$p \to 1$ is a knapsack problem (in fact, it is the trivial instance in which the $i\th$ item
has weight $1$ and utility $u_i=\trace M_i$).
Note that a matrix $M_i$ may be chosen $n_i$ times in Problem~\eqref{Pp_intro},
while choosing a matrix more than once in Problem~\eqref{maxrank_intro} cannot increase
the rank. Therefore we also define the binary variant of Problem~\eqref{Pp_intro}:
\begin{equation}
\max_{\vec{n}}  \left\{\varphi_p \left( \vec{n} \right):\ \vec{n}\in\{0,1\}^s,\ \sum_{i\in[s]} n_i \leq N\right\}
\tag{$P_p^\textrm{bin}$} \label{Ppbin}
\end{equation}
We shall also consider the case in which the selection of the $i\th$ matrix
costs $c_i$, and a total budget $B$ is allowed. This is the budgeted version of
the problem:
\begin{equation}
\max_{\vec{n}}  \left\{\varphi_p \left( \vec{n} \right):\ \vec{n}\in\mathbb{N}^s,\ \sum_{i\in[s]} c_i n_i \leq B\right\}
\tag{$P_p^\textrm{bdg}$} \label{Ppbdg}
\end{equation}

Throughout this article, we use the term \emph{design} for
the variable $\vec{n}=(n_1,\ldots,n_s)\in\mathbb{N}^s$.
We say that $\vec{n}$ is a $N-$\emph{replicated design}
if it is feasible for Problem~\eqref{Pp_intro},
a $N-$\emph{binary design}
if $\vec{n}$ is feasible for Problem~\eqref{Ppbin},
and a $B-$\emph{budgeted design} when it satisfies the constraints of~\eqref{Ppbdg}.

\subsection{Motivation: optimal experimental design}
The theory of \emph{optimal design of experiments} plays a central role in statistics.
It studies how to best select experiments in order to estimate a set of parameters.
Under classical assumptions,
the best linear unbiased estimator is given by least square theory, and lies within confidence ellipsoids which are described
by a positive semidefinite matrix depending only on the selected experiments. The \emph{optimal design of experiments}
aims at selecting the experiments in order to make these confidence ellipsoids as small as possible, which
leads to more accurate estimators. 

A common approach consists in minimizing a
scalar function measuring these ellipsoids, 
where the function is taken from the class of $\Phi_p$-information functions
proposed by Kiefer~\cite{Kief75}. This leads to a combinatorial optimization problem
(decide how many times each experiment should be performed) involving
a spectral function which is applied to the information matrix of the experiments.
For $p\in ]0,1]$, the Kiefer's $\Phi_p$-optimal design problem is
equivalent to Problem~\eqref{Pp_intro} (up to the exponent $1/p$ in the objective function).

In fact, little attention has been given to the combinatorial aspects of Problem~\eqref{Pp_intro}
in the optimal experimental design literature.
The reason is that
there is a natural relaxation of the problem which is much more tractable and
usually yields very good results: instead of
determining the exact number of times $n_i$ that each experiment will be selected,
the optimization is done over the fractions $w_i=n_i/N\in[0,1]$,
which reduces the problem to the maximization of a concave function over a convex set
(this is the theory of \textit{approximate optimal designs}).
For the common case, in which the number $N$ of experiments to perform is large
and $N>s$ (where $s$ is the number of available experiments), this approach
is justified by a result of Pukelsheim and Rieder~\cite{PR92},
who give a rounding procedure to transform an optimal approximate design
$\vec{w^*}$ into an
$N-$replicated design $\vec{n}=(n_1,\ldots,n_s)$ which approximates the optimum
of the Kiefer's $\Phi_p-$optimal
design problem within a factor $1-\frac{s}{N}$.

\sloppy{The present developments were motivated by a joint work with Bouhtou
and Gaubert~\cite{BGSagnol08Rio,SagnolGB10ITC} on the application of optimal
experimental design methods to the identification of the traffic in an Internet
backbone.} This problem describes an \textit{underinstrumented situation}, in which
a small number $N<s$ of experiments should be selected. In this case,
the combinatorial aspects of Problem~\eqref{Pp_intro} become crucial.
A similar problem was studied by Song, Qiu and Zhang~\cite{SQZ06}, who proposed to
use a greedy algorithm to approximate the solution of Problem~\eqref{Pp_intro}.
In this paper, we give an approximation bound which justifies this approach. 
Another question addressed in this manuscript is whether it is
appropriate to take roundings of (continuous) approximate designs in the
underinstrumented situation (recall that this is the common approach
when dealing with experimental design problems in the \emph{overinstrumented} case,
where the number $N$ of experiments is large when compared to $s$).

Appendix~\ref{sec:problemStatement} is devoted to the application
to the theory of optimal experimental designs; we explain how
a statistical problem (choose which experiments to conduct in order
to estimate a set of parameters) leads to the study
of Problem~\eqref{Pp_intro},
with a particular focus to the \emph{underinstrumented} situation
described above.
For more details on the subject, the reader is referred to the monographs of Fedorov~\cite{Fed72} and Pukelsheim~\cite{Puk93}.

\subsection{Organisation and contribution of this article}
~\par

\begin{table}[ht!]
{\small
\begin{center}
 \begin{tabular}{lll}
  Algorithm & Approximation factor for Problem~\eqref{Pp_intro} & Reference\\
\hline
& \vspace{-4mm} &\\
  Greedy & $1-e^{-1} \qquad$ (or $1-(1-\frac{1}{N})^N$) &
  \ref{coro:1m1se} (\cite{NWF78})\vspace{3mm}\\
Any $N-$replicated design $\vec{n}$&
\multirow{2}{*}{$\frac{1}{N} \sum_{i=1}^s n_i^p (w_i^*)^{1-p}$} &
\multirow{2}{*}{\ref{prop:boundW_n}}\\
(posterior bound) & & \vspace{3mm}\\

\hspace{-2mm}\begin{tabular}{l}Rounding~\ref{algo:greedyrounding} \\ (prior bound) \end{tabular} &
{\footnotesize $\left\{\begin{array}{cl}
           \left(\frac{N}{s} \right)^{1-p} & \quad \textrm{if } \left(\frac{N}{s} \right)^{1-p} \leq \frac{1}{2-p} ; \\
     1-\frac{s}{N}(1-p)\left(\frac{1}{2-p}\right)^{\frac{2-p}{1-p}} & \quad
    \textrm{Otherwise} \end{array} \right.$}
 & \ref{theo:factorF} \vspace{3mm}\\
Apportionment rounding & $(1-\frac{s}{N})^p \qquad$ if $N\geq s$ & \cite{PR92}\vspace{1cm}\\

  Algorithm & Approximation factor for Problem~\eqref{Ppbin} & Reference\\
\hline
& \vspace{-4mm} &\\
Greedy & $1-e^{-1} \qquad $ (or $1-(1-\frac{1}{N})^N$) & \ref{coro:1m1se} (\cite{NWF78}) \vspace{3mm}\\
 Any $N-$binary design $\vec{n}$ &
\multirow{2}{*}{$\frac{1}{N}\sum_{i=1}^s n_i (w_{i}^*)^{1-p}$} &
\multirow{2}{*}{\ref{prop:boundW}}\\
(posterior bound)& & \vspace{3mm}\\
\hspace{-2mm}\begin{tabular}{l}Keep the $N$ largest coord. \\ of $\vec{w^*}$ (prior bound) \end{tabular}
& $ \left( \frac{N}{s} \right)^{1-p}$ \qquad if $p\leq 1-\frac{\ln N}{\ln s}$ &
 \ref{theo:factorFbin} \vspace{1cm} \\

Algorithm & Approximation factor for Problem~\eqref{Ppbdg} & Reference\\
\hline
& \vspace{-4mm} &\\
Adapted Greedy & $1-e^{-\beta}\simeq 0.35$ (where $e^\beta=2-\beta$) & \ref{rem:budg}(\cite{Wol82}) \vspace{3mm}\\
Greedy+triples enumeration & $1-e^{-1} \qquad $ & \ref{rem:budg}(\cite{Svi04}) \vspace{3mm}\\
Any $B-$budgeted design $\vec{n}$ &
\multirow{2}{*}{$\frac{1}{B}\sum_{i=1}^N c_i n_i^p (w_{i}^*)^{1-p}$} &
\multirow{2}{*}{\ref{prop:boundW_bdg}}\\
(posterior bound) & & \vspace{0mm}
 \end{tabular}
\end{center}
}
\caption{Summary of the approximation bounds obtained in this paper, as
well as the bound of Pukelsheim and Rieder~\cite{PR92}.
The column ``Reference''
indicates the number of the theorem, proposition or remark where the bound is proved
(a citation in parenthesis means a direct application of a
result of the cited paper, which is possible thanks to the
submodularity of $\varphi_p$ proved in Corollary~\ref{coro:SubmodExample}).
In the table, 
\emph{posterior} denotes a bound which depends on the continuous solution
$\vec{w^*}$ of the relaxed problem, while a \emph{prior bound} depends only on
the parameters of the problem.
\label{tab:factors}}
\end{table}

The objective of this article is to study some approximation algorithms for the class
of problems~\eqref{Pp_intro}${}_{p \in [0,1]}$. Several results presented in this article were
already announced in the companion papers~\cite{BGSagnol08Rio,BGSagnol10ENDM}, without
the proofs. This paper provides all the proofs of the results of~\cite{BGSagnol10ENDM}
and gives new results for the rounding algorithms.
We shall now present the contribution
and the organisation of this article.

In Section~\ref{sec:submod}, we
establish a matrix inequality (Proposition~\ref{prop:ineqPropos})
which shows that a class of spectral functions is submodular (Corollary~\ref{coro:Submod}).
As a particular case of the latter result, the objective function
of Problem~\eqref{Pp_intro} is submodular for all $p\in[0,1]$.
The submodularity of this class of spectral functions is an original contribution of
this article for $0<p<1$, however a particular case of this result was announced --without a proof-- in the companion paper on the telecom application~\cite{BGSagnol08Rio}.
In the limit case $p=0$, we obtain two functions
which were already known to be submodular (the rank and the log of determinant
of a sum of matrices).

Due to a celebrated result of Nemhauser, Wolsey
and Fisher~\cite{NWF78}, the submodularity of the criterion implies that the greedy approach,
which has often been used for this problem,
always gives a design within $1-e^{-1}$ of the optimum (Theorem~\ref{coro:1m1se}).
We point out that the submodularity of the determinant criterion was noticed earlier
in the optimal experimental design literature,
but under an alternative form~\cite{RS89}: Robertazzi and Schwartz showed that the determinant
of the inverse of a sum of matrices is supermodular, and they used it
to write an algorithm for the construction of approximate
designs (i.e. without integer variables)
which is based on the accelerated greedy algorithm of Minoux~\cite{Min78}.
In contrast, the originality of the present paper
is to show
that a whole class of criteria satisfies the submodularity property,
and to study the consequences in
terms of approximability of a combinatorial optimization problem.

In Section~\ref{sec:rounding}, we investigate the legitimacy of using rounding algorithms
to construct a $N-$replicated design $\vec{n}=(n_1,\ldots,n_s)\in\mathbb{N}^s$
or a $N$-binary design $\vec{n}\in\{0,1\}^s$
from an optimal approximate design $\vec{w^*}$,
i.e. a solution of a continuous relaxation of Problem~\eqref{Pp_intro}. 
We establish an inequality (Propositions~\ref{prop:boundW} and~\ref{prop:boundW_n})
which bounds from below the approximation ratio of any integer design, by a function
which depends on the continuous solution $\vec{w^*}$. Interestingly, this
inequality generalizes a classical result from the theory of optimal designs
(the upper bound on the weights of a D-optimal design~\cite{Puk80,HT09}
is a particular case ($p=0$) of Proposition~\ref{prop:boundW}).
The proof of this result is presented in Appendix~\ref{sec:proofIneq} ; it relies on
matrix inequalities and several
properties of the differentiation of a scalar function applied to symmetric matrices.
Then we point out that the latter lower bound can be maximized by an incremental
algorithm which is well known in the resource allocation community
(Algorithm~\ref{algo:greedyrounding}),
and we derive approximation bounds
for Problems~\eqref{Pp_intro} and~\eqref{Ppbin} which do not depend on
$\vec{w^*}$ (Theorems~\ref{theo:factorFbin} and~\ref{theo:factorF}).
For the problem with replicated designs~\eqref{Pp_intro}, the approximation
factor is an increasing function of $p$ which tends to $1$ as $p\to1$.
In many cases, the
approximation guarantee for designs obtained by rounding is better than
the greedy approximation factor $1-e^{-1}$.

We have summarized in Table~\ref{tab:factors} the approximation
results proved in this paper (this table also includes another
known approximability result for Problem~\eqref{Pp_intro}, 
the \emph{efficient apportionment rounding} of 
Pukelsheim and Rieder~\cite{PR92}).

\section{Submodularity and Greedy approach} \label{sec:submod}
In this section, we study the greedy algorithm for solving Problems~\eqref{Pp_intro}
and~\eqref{Ppbin}
through the submodularity of $\varphi_p$.
We first recall a result presented in~\cite{BGSagnol08Rio}, which states
that the \textit{rank optimization} problem is NP-hard, by a reduction from the \emph{Maximum Coverage}
problem. It follows that for all positive $\varepsilon$, there is no polynomial-time
algorithm which approximates~\eqref{maxrank_intro} by a factor of $1-\frac{1}{e}+\varepsilon$ unless $P=NP$
(this has been proved by Feige for the Maximum Coverage problem~\cite{Fei98}).
Nevertheless, we show that this bound
is the worst possible ever, and that the greedy algorithm always
attains it.

To this end, we show that a class of spectral functions (which includes 
the objective function of Problem~\eqref{Pp_intro}) is \emph{nondecreasing submodular}.
The maximization of submodular functions over a matroid has been
extensively studied~\cite{NWF78,CC84,CCPV07,Von08,KST09}, 
and we shall use known approximability results.

To study its approximability, we can think of Problem~\eqref{Pp_intro} as the maximization of a set function
$\varphi_p':2^E \mapsto \mathbb{R}^+$. To this end, note that each design
$\vec{n}$ can be seen as a subset of $E$, where
$E$ is a pool which contains $N$ copies of each experiment (this allows
us to deal with replicated designs, i.e.\ with experiments that are conducted several times;
if replication is not allowed (Problem~\eqref{Ppbin}), we simply set $E:=[s]$).
Now, if $S$ is a subset of $E$ corresponding to the design $\vec{n}$, we 
define $\varphi_p'(S):=\varphi_p(\vec{n})$. In the sequel, we identify the set function
$\varphi_p'$ with $\varphi_p$ (i.e., we omit the \emph{prime}).

We also point out that multiplicative approximation factors for the $\Phi_p-$optimal problem cannot be considered
when $p\leq0$, since the criterion is identically $0$ as
long as the the information matrix is singular. For $p\leq0$ indeed, the instances of the $\Phi_p$-optimal problem where no feasible
design lets $M_F(\vec{n})$ be of full rank have an optimal value of $0$. For all the other instances, any
polynomial-time algorithm
with a positive approximation factor would necessarily return a design of full rank. Provided that $P\neq NP$, this would contradict the
NP-hardness of
\emph{Set-Cover} (it is easy to see that \emph{Set Cover} reduces to the problem
of deciding whether there exists a set $S$ of cardinal $N$ such that $\sum_{i\in S} M_i$
has full rank for some diagonal matrices $M_i$, by a similar argument to the one given in the first paragraph of this article).
Hence, we investigate approximation algorithms only in the case $p \in [0,1]$.

\subsection{A class of submodular spectral functions}

In this section, we are going to show that a class of spectral functions is submodular.
We recall that a real valued function $F:2^E \rightarrow \mathbb{R}$, defined on every subset of $E$ is called nondecreasing
if for all subsets $I$ and $J$ of $E$,\ $I \subseteq J$ implies  $F(I)\leq F(J)$. We also give
the definition of a \emph{submodular} function: 

\begin{definition}[Submodularity] A real valued set function $F$ : $2^E \longrightarrow \mathbb{R}$
is \emph{submodular} if it satisfies the following condition :
$$F(I)+F(J)\geq F(I\cup J)+F(I\cap J) \quad \textrm{for all}\quad I,J\subseteq E.$$
\end{definition}

We next recall the definition of operator monotone functions. The latter are
real valued functions applied to hermitian matrices: if $A=U\operatorname{Diag}(\lambda_1,\ldots,\lambda_m)U^*$ is
a $m\times m$ hermitian matrix (where $U$ is unitary and
$U^*$ is the conjugate of $U$), the matrix
$f(A)$ is defined as $U\operatorname{Diag}(f(\lambda_1),\ldots,f(\lambda_m))U^*$.

\begin{definition}[Operator monotonicity] 
A real valued function $f$ is \emph{operator monotone}
on $\RR_+$ (resp.\ $\RR_+^*$) if for every pair of positive semidefinite (resp.\ positive definite) matrices $A$ and $B$,
$$A\preceq B \Longrightarrow f(A)\preceq f(B).$$
We say that $f$ is \emph{operator antitone} if $-f$ is operator monotone.
\end{definition}

The next proposition is a matrix inequality of independent interest; it will be useful to show that $\varphi_p$ is submodular.
Interestingly, it can be seen as
an extension of the Ando-Zhan Theorem~\cite{AZ99}, which reads as follows: \textit{Let $A$, $B$ be semidefinite positive matrices. For any unitarily invariant norm
$\normtr \cdot \normtr$, and for every non-negative operator monotone function $f$ on $[0,\infty)$,
$$ \normtr f(A+B) \normtr \leq \normtr f(A)+f(B) \normtr.$$}
Kosem~\cite{Kos06} asked whether it is possible to extend this inequality as follows:
$$ \normtr f(A+B+C) \normtr \leq \normtr f(A+B)+f(B+C)-f(C) \normtr,$$
and gave a counterexample involving the trace norm and the function $f(x)=\frac{x}{x+1}$.
However, we show in next proposition that the previous inequality holds for the trace norm and
every primitive $f$ of an operator antitone function (in particular,
for $f(x)=x^p,\ p\in]0,1]$). Note that the previous inequality is not true 
for any unitarily invariant norm and $f(x)=x^p$ either. It is easy to find counterexamples with the spectral radius norm.

\begin{proposition} \label{prop:ineqPropos} Let $f$ be a real function defined on $\RR_+$ and differentiable on $\RR_+^*$.
If $f'$ is operator antitone on $\RR_+^*$, then for all triples $(X,Y,Z)$ of $m\times m$ positive semidefinite matrices,
\begin{align}
 \mathrm{trace}\ f(X+Y+Z)+ \mathrm{trace}\ f(Z) \leq \mathrm{trace}\ f(X+Z) + \mathrm{trace}\ f(Y+Z). \label{ineqxyz}
\end{align}
\end{proposition}

\begin{proof}
Since the eigenvalues of a matrix are continuous functions of its entries,
and since $\mathbb{S}_m^{++}$ is dense in $\mathbb{S}_m^{+}$,
it suffices to establish the inequality when $X$, $Y$, and $Z$ are 
positive definite. Let $X$ be an arbitrary positive definite matrix. We consider the map:
\begin{align}
\psi :  \mathbb{S}_m^{+} & \longrightarrow \mathbb{R} \nonumber \\
T & \longmapsto \mathrm{trace}\ f(X+T) - \mathrm{trace}\ f(T). \nonumber
\end{align}
The inequality to be proved can be rewritten as
$$\psi(Y+Z)\leq \psi(Z).$$
We will prove this by showing that $\psi$ is nonincreasing with respect to the L\"owner ordering in
the direction generated by any positive semidefinite matrix. To this end, we compute the Frechet derivative
of $\psi$ at $T \in \mathbb{S}_m^{++}$ in the direction of an arbitrary matrix $H \in \mathbb{S}_m^{+}$. By definition,
\begin{equation}
D\psi(T)(H)=\lim_{\epsilon\rightarrow 0} \frac{1}{\epsilon} \big( \psi(T+\epsilon H)-\psi(T) \big). \label{Dfrechet}
\end{equation}
When $f$ is an analytic function, $X \longmapsto \mathrm{trace}\ f(X)$ is Frechet-differentiable,
and an explicit form of the derivative is known~(see~\cite{HP95,JB06}):
$D\big(\mathrm{trace}\ f(A)\big)(B)=\mathrm{trace} \big(f'(A)B\big)$.
Since $f'$ is operator antitone on $\RR_+^*$, a famous result of L\"owner~\cite{Low34}
tells us (in particular) that $f'$ is analytic at all points of the positive real axis,
and the same holds for $f$.
Provided that the matrix $T$ is positive definite (and hence $X+T$), we have
$$D\psi(T)(H)=\mathrm{trace} \Big(\   \big(f'(X+T)-f'(T)\big) H \Big).$$
By antitonicity of $f'$ we know that the matrix $W=f'(X+T)-f'(T)$ is negative semidefinite.
For a matrix $H\succeq0$, we have therefore:
\begin{align*}
 D\psi(T)(H) & =\trace\ (WH) \leq 0.
\end{align*}
Consider now $h(s):=\psi(sY+Z)$. For all $s \in [0,1]$, we have $$h'(s)=D\psi(sY+Z)(Y)\leq 0,$$ and so, $h(1)=\psi(Y+Z)\leq h(0)=\psi(Z)$, 
from which the desired inequality follows.
\end{proof}

\begin{corollary} \label{coro:Submod}
Let $M_1,\ldots,M_s$ be $m\times m$ positive semidefinite matrices.
If $f$ satisfies the assumptions of Proposition~\ref{prop:ineqPropos},
then the set function
$F: 2^{[s]} \rightarrow \mathbb{R}$ defined by
$$\forall I \subset [s],\ F(I)=\trace\ f(\sum_{i\in I} M_i),$$
is submodular.
\end{corollary}

\begin{proof}

Let $I,J \subseteq 2^{[s]}$. We define 
\begin{align} X=\sum_{i \in I\setminus J} M_{i},\ Y=\sum_{i\in J\setminus I} M_{i},\ Z=\sum_{i\in I\cap J} M_{i}. \nonumber
\end{align}
It is easy to check that
\begin{align} 
F(I)&=\textrm{trace}\ f(X+Z), \nonumber\\
F(J)&=\textrm{trace}\ f(Y+Z), \nonumber\\
F(I\cap J)&=\textrm{trace}\ f(Z), \nonumber\\
F(I\cup J)&=\textrm{trace}\ f(X+Y+Z). \nonumber
\end{align}
Hence, Proposition~\ref{prop:ineqPropos}  proves the submodularity of $F$.\end{proof}

A consequence of the previous result is that the objective function
of Problem~\eqref{Pp_intro} is submodular. In the limit case $p\to 0^+$,
we find two well-known submodular functions:
\begin{corollary} \label{coro:SubmodExample}
 Let $M_1,...,M_s$ be $m\times m$ positive semidefinite matrices.
\begin{itemize}
 \item[(i)] $\forall p\in]0,1], I \mapsto \trace (\sum_{i\in I}M_i)^p$ is submodular.
\item[(ii)] $I \mapsto \rank (\sum_{i\in I}M_i)$ is submodular.
\end{itemize}
If moreover every $M_i$ is positive definite, then:
\begin{itemize}
\item[(iii)] $I \mapsto \log\det (\sum_{i\in I}M_i)$ is submodular.
\end{itemize}
\end{corollary}

\begin{proof}
It is known that $x\mapsto x^q$ is operator antitone
on $\mathbb{R}_+^*$ for all $q\in[-1,0[$. Therefore, the
derivative of the function $x\mapsto x^p$ (which is $px^{p-1}$), is operator antitone on $\mathbb{R}_+^*$ for all $p\in]0,1[$.
This proves the point $(i)$ for $p\neq 1$. The case $p=1$ is trivial, by linearity of the trace.

The submodularity of the rank $(ii)$ and of $\log\det$ $(iii)$ are classic. 
Interestingly, they are obtained
as the limit case of $(i)$ as $p\rightarrow0^+$. (For $\log\det$, we must consider the second term in the asymptotic
development of $X \mapsto \trace\ X^p$ as $p$ tends to $0^+$, cf.\ Equation~\eqref{expansion1}).
\end{proof}

\subsection{Greedy approximation}

We next present some consequences of the submodularity of
$\varphi_p$ for the approximability of Problem~\eqref{Pp_intro}.
Note that the
results of this section hold in particular for $p=0$,
and hence for the \textit{rank maximization}
problem~\eqref{maxrank_intro}. They also hold for $E=[s]$, i.e.\ for Problem~\eqref{Ppbin}.
We recall that the principle of the greedy algorithm
is to start from $\mathcal{G}_0=\emptyset$
and to construct sequentially the sets 
$$\mathcal{G}_{k+1}:=\mathcal{G}_{k} \cup \displaystyle{\operatorname{argmax}_{i \in E\setminus\mathcal{G}_k}}\ \varphi_p(\mathcal{G}_k \cup \{i\}),$$
until $k=N$.

\begin{theorem}[Approximability of Problem~\eqref{Pp_intro}] \label{coro:1m1se}
Let $p\in[0,1]$. The greedy algorithm always yields a
solution within a factor $1-\frac{1}{e}$  of the optimum of Problem~\eqref{Pp_intro}.
\end{theorem}

\begin{proof} We know from Corollary~\ref{coro:SubmodExample} that for all $p\in[0,1]$, $\varphi_p$
is submodular ($p=0$ corresponding to the rank maximization problem).
In addition, the function $\varphi_p$ is nondecreasing, because $X \longrightarrow X^p$ is a matrix
monotone function for $p \in [0,1]$ (see e.g.~\cite{Zhan02}) and $\varphi_p(\emptyset)=0$.

Nemhauser, Wolsey and Fisher~\cite{NWF78} proved the result of this theorem for any nondecreasing
submodular function $f$ satisfying $f(\emptyset)=0$ which is maximized over a uniform matroid.
Moreover when the maximal number of matrices which can be selected is $N$, this approximability ratio can
be improved to $1-\big(1-1/N\big)^N.$
 \end{proof}

\begin{remark}
One can obtain a better bound by considering the \textit{total curvature} of a given instance,
which is defined by:
$$c=\max_{i \in [s]}\quad 1-\frac{\varphi_p\big(E)-\varphi_p\big(E\setminus\{i\}\big)}{\varphi_p\big(\{i\}\big)} \in [0,1].$$
Conforti and Cornuejols~\cite{CC84} proved that the greedy algorithm always achieves
a \mbox{$\frac{1}{c}\big(1-(1-\frac{c}{N})^N \big)$-approximation} factor for
the maximization of an arbitrary nondecreasing submodular function with total curvature $c$.
In particular, since $\varphi_1$ is additive
it follows that the total curvature for $p=1$ is $c=0$, yielding an approximation factor of $1$:
$$\lim_{c \rightarrow 0^+} \frac{1}{c}\big(1-(1-\frac{c}{N})^N \big)=1.$$ 
As a consequence, the greedy algorithm always gives the optimal solution of the problem. Note that Problem~$(P_1)$ is
nothing but a \emph{knapsack} problem, for which it is well known that the greedy algorithm is optimal
if each available item has the same weight.
However, it is not possible to give an upper bound on the total curvature $c$ for
other values of $p \in [0,1[$, and $c$ has to be computed for each instance.
\end{remark}

\begin{remark} \label{rem:budg}
The problem of maximizing a nondecreasing submodular function
subject to a budget constraint of the form
$\sum_i c_i n_i \leq B$, where $c_i\geq 0$
is the cost for
selecting the element $i$ and $B$ is the total allowed budget,
has been studied by several authors.
Wolsey presented an adapted greedy algorithm~\cite{Wol82} with
a proven approximation guarantee of $1-e^{-\beta}\simeq0.35$,
where $\beta$ is the unique root of the equation $e^x=2-x$.
More recently, Sviridenko~\cite{Svi04} showed that the budgeted
submodular maximization problem was still \mbox{$1-1/e-$approximable}
in polynomial time,
with the help of an algorithm which associates the greedy with a
partial  enumeration of every solution of cardinality $3$.
\end{remark}

%
%

We have attained so far an approximation factor of $1-e^{-1}$ for all $p\in [0,1[$, while
we have a guarantee of optimality of the greedy algorithm for $p=1$. This leaves a
feeling of mathematical dissatisfaction, since intuitively the problem should be easy when $p$
is very close to $1$. In the next section we remedy to this problem, by giving a rounding algorithm
with an approximation factor $F(p)$ which depends on $p$, and such that $p\mapsto F(p)$ is continuous, nondecreasing and
$\lim_{p\to 1} F(p)=1$.

\section{Approximation by rounding algorithms} \label{sec:rounding}

The optimal design problem has a natural continuous relaxation which is simply
obtained by removing the integer constraint on the design variable $\vec{n}$,
and has been extensively studied~\cite{Atw73,DPZ08,Yu10a,Sagnol09SOCP}.
As mentioned in the introduction, several authors proposed to solve this continuous
relaxation and to round the solution to obtain a near-optimal discrete design.
While this process is well understood when $N\geq s$, we are not aware of
any bound justifying this technique in the underinstrumented situation $N<s$.

\subsection{A continuous relaxation}

The continuous relaxation of Problem~\eqref{Pp_intro} which we consider is obtained
by replacing the integer variable
$\vec{n}\in \mathbb{N}^s$ by a continuous variable $\vec{w}$
in Problem~\eqref{optdesProblem}:
\begin{equation}
\max_{\substack{\vec{w}\ \in (\RR_+)^s\\ \sum_k w_k \leq N}}\ \Phi_p(M_F(\vec{w})) \label{Pcont}
\end{equation}

Note that the criterion $\varphi_p(\vec{w})$ is raised to the power $1/p$ in Problem~\eqref{Pcont}
(we have $\Phi_p(M_F(\vec{w}))= m^{-1/p}\varphi_p(\vec{w})^{1/p}$ for $p>0$).
The limit of Problem~\eqref{Pcont} as $p\to 0^+$ is hence the maximization of the
determinant of $M_F(\vec{w})$ (cf.\ Equation~\eqref{phip}). 

We assume without loss of generality that the matrix $M_F(\indic)=\sum_{k=1}^s M_k$ is of
full rank (where $\indic$ denotes the vector of all ones). This ensures the existence of a
vector $\vec{w}$ which is feasible for Problem~\eqref{Pcont}, and such that $M_F(\vec{w})$ has full rank.
If this is not the case ($r^*:=\rank(M_F(\indic))<m)$, we define instead a projected
version of the continuous relaxation: Let $U \Sigma U^T$ be a singular value decomposition of
$M_F(\indic)$. We denote by $U_{r^*}$ the matrix formed with the $r^*$ leading singular vectors of $M_F(\indic)$,
i.e.\ the $r^*$ first columns of $U$.
It can be seen that Problem~\eqref{Pcont}
is equivalent to the problem with projected information matrices $\bar{M_k}:=U_{r^*}^T M_k U_{r^*}$
(see Paragraph $7.3$ in~\cite{Puk93}).

The functions $X\mapsto \log(\det(X))$ and $X\mapsto X^p$ ($p\in]0,1]$) are
strictly concave on the interior of $\mathbb{S}_m^+$, so that the continuous relaxation~\eqref{Pcont}
can be solved by interior-points technique or multiplicative
algorithms~\cite{Atw73,DPZ08,Yu10a,Sagnol09SOCP}. The strict concavity of the objective function indicates in addition that
Problem~\eqref{Pcont} admits a unique solution if and only if
$$w_1 M_1+w_2 M_2 + \ldots + w_s M_s=y_1 M_1 + y_2 M_2 + \ldots + y_s M_s \Rightarrow (w_1,\ldots,w_s)=(y_1,\ldots,y_s),$$
that is to say whenever the matrices $M_i$ are linearly independent.
In this paper, we focus on the rounding techniques only,
and we assume that an optimal solution $\vec{\vec{w^*}}$ of the
relaxation~\eqref{Pcont} is already known.
In the sequel, we also denote a discrete solution of Problem~\eqref{Pp_intro}
by $\vec{n}^*$ and a binary solution of Problem~\eqref{Ppbin} by $S^*$.
Note that we always have $\varphi_p(\vec{w^*}) \geq \varphi_p(\vec{n}^*)\geq \varphi_p(S^*)$.

\subsection{Posterior bounds}

In this section, we are going to bound from below
the approximation ratio $\varphi_p(\vec{n})/\varphi_p(\vec{w^*})$
for an arbitrary discrete design $\vec{n}$,
and we propose a rounding algorithm which maximizes this approximation factor.
The lower bound depends on the continuous optimal variable~$\vec{w^*}$,
and hence we refer it as a \emph{posterior} bound.
We start with a result for binary designs ($\forall i\in [s], n_i\leq 1$),
which we associate with a subset $S$ of $[s]$ as in Section~\ref{sec:submod}.
The proof relies on several matrix inequalities and technical lemmas on the directional derivative of
a scalar function applied to a symmetric matrix, and is therefore presented in Appendix~\ref{sec:proofIneq}.

\begin{proposition} \label{prop:boundW}
Let $p\in[0,1]$ and $\vec{w^*}$ be optimal for the continuous relaxation~\eqref{Pcont}
of Problem~\eqref{Pp_intro}. Then, for any subset $S$ of $[s]$, the following inequality holds:
$$\frac{1}{N} \sum_{i\in S} (w_i^*)^{1-p} \leq \frac{\varphi_p(S)}{\varphi_p(\vec{w^*})}.$$
\end{proposition}

\begin{remark}
In this proposition and in the remaining of this article,
we adopt the convention $0^0=0$. 
\end{remark}

We point out that this proposition includes as a special case a result of Pukelsheim~\cite{Puk80}, already generalized by Harman and Trnovsk\'a~\cite{HT09},
who obtained:
$$\frac{w_i^*}{N} \leq \frac{\rank M_i}{m},$$
i.e.\ the inequality of Proposition~\ref{prop:boundW} for $p=0$ and a singleton $S=\{i\}$.
However the proof is completely different in our case.
Note that there is no constraint of the form $w_i\leq 1$ in the continuous
relaxation~\eqref{Pcont}, although the previous proposition
relates to binary designs $S\in[s]$.
Proposition~\ref{prop:boundW} suggests to select the $N$ matrices
with the largest coordinates $w_i^*$ to obtain a candidate
$S$ for optimality of the binary problem~\eqref{Ppbin}. We will
give in the next section a \emph{prior bound} (i.e., which does not depend on $\vec{w^*}$)
for the efficiency of this rounded design.\par
\vspace{5mm}

We can also extend the previous proposition to the case of replicated designs $\vec{n} \in \mathbb{N}^s$
(note that the following proposition does not require the design
$\vec{n}$ to satisfy $\sum_i n_i=N$):

\begin{proposition} \label{prop:boundW_n}
Let $p\in[0,1]$ and $\vec{w^*}$ be optimal for the continuous relaxation~\eqref{Pcont}
of Problem~\eqref{Pp_intro}. Then, for any design $\vec{n}\in\mathbb{N}^s$, the following inequality holds:
$$\frac{1}{N} \sum_{i\in[s]} n_i^p (w_i^*)^{1-p} \leq \frac{\varphi_p(\vec{n})}{\varphi_p(\vec{w^*})}.$$
\end{proposition}

\begin{proof}
We consider the problem in which the matrix $M_i$ is replicated $n_i$ times:
$$\forall i\in[s],\ \forall k\in[n_i], M_{i,k}=M_i.$$
Since $\vec{w^*}$ is optimal for Problem~\eqref{Pcont}, it is clear that
$(w_{i,k})_{(i,k) \in \cup_{j\in[s]} \{j\} \times [n_j]}$ is optimal
for the problem with replicated matrices if
\begin{equation}
\forall i \in [s], \sum_{k \in [n_i]} w_{i,k} = w_i^*, \label{constraintalloc} 
\end{equation}
i.e.\ $w_{i,k}$ is the part of $w_i^*$
allocated to the $k\th$ copy of the matrix~$M_i$. For such a vector,
Proposition~\ref{prop:boundW} shows that
$$\frac{\varphi_p(\vec{n})}{\varphi_p(\vec{w^*})} \geq \frac{1}{N} \sum_{i=1}^s \sum_{k=1}^{n_i} (w_{i,k}^*)^{1-p}.$$
Finally, it is easy to see (by concavity) that the latter lower bound is maximized with respect
to the constraints of Equation~\eqref{constraintalloc} if
$\forall i\in[s], \forall k\in [n_i],\ w_{i,k} = \frac{w_i^*}{n_i}$:
$$\frac{\varphi_p(\vec{n})}{\varphi_p(\vec{w^*})} \geq
\frac{1}{N} \sum_{i=1}^s \sum_{k=1}^{n_i} \left(\frac{w_i^*}{n_i}\right)^{1-p}
=\frac{1}{N} \sum_{i=1}^s n_i^p (w_i^*)^{1-p}.$$
\end{proof}

We next give a simple rounding algorithm which finds the feasible
design $\vec{n}$ which maximizes the lower bound of Proposition~\ref{prop:boundW_n}:
\begin{equation}
 \max_{\substack{\vec{n} \in \mathbb{N}^s \\ \sum n_i = N }}\quad \sum_{j \in [s]} n_j^p\ w_j^{1-p}. \label{probaposteriori}
\end{equation}
The latter maximization problem is in fact a \emph{ressource allocation problem
with a convex separable objective},
and the incremental algorithm which we give below is
well known in the resource allocation community (see e.g.~\cite{IK88}).

\begin{algorithm}
\caption{[Incremental rounding]
\label{algo:greedyrounding}}
\begin{algorithmic}
\STATE \textbf{Input:} A nonnegative vector $\vec{w} \in \mathbb{R}^s$ such that $\sum_{i=1}^s w_i =N\in \mathbb{N}\setminus\{0\}$.
\STATE Sort the coordinates of $\vec{w}$; We assume wlog that
$w_1\geq w_2\geq\ldots\geq w_s$;
\STATE $\vec{n} \leftarrow [1,0\ldots,0]\in\mathbb{R}^s$
\FOR{$k=2\ldots N$}
\STATE Select an index $i_{max} \in \operatorname{argmax}_{i\in[s]} \limits \big((n_i+1)^p- n_i^p\big)\ w_i^{1-p}$
\STATE $n_{i_{max}} \leftarrow n_{i_{max}}+1$
\ENDFOR
\STATE \textbf{return:} a $N-$replicated design $\vec{n}$ which maximizes $\sum_{i=1}^s n_i^p w_i^{1-p}$.
\end{algorithmic}
\end{algorithm}

\begin{remark}
If $\vec{w}$ is sorted ($w_1\geq w_2\geq\ldots\geq w_s$), then
the solution of Problem~\eqref{probaposteriori} clearly satisfies
$n_1 \geq n_2 \geq \ldots \geq n_s$. Consequently, it is not necessary to
test every index $i \in [s]$ to compute the $\operatorname{argmax}$ in
Algorithm~\ref{algo:greedyrounding}. Instead, one only needs to 
compute the increments $\big((n_i+1)^p- n_i^p\big)\ w_i^{1-p}$
for the $i\in [s]$ such that
$i=1$ or $n_i+1 \leq n_{i-1}$.
\end{remark}

We shall now give a posterior bound for the budgeted problem~\eqref{Ppbdg}. We only
provide a sketch of the proof, since the reasoning is the same as for
Propositions~\ref{prop:boundW} and~\ref{prop:boundW_n}.
We also point out
that the approximation bound provided in the next proposition
can be maximized over the set of $B-$budgeted designs,
thanks to a dynamic programming algorithm
which we do not detail here (see~\cite{MM76}).

\begin{proposition} \label{prop:boundW_bdg}
Let $p\in[0,1]$ and $\vec{w^*}$ be optimal for the continuous relaxation
\begin{equation}
\max_{\vec{w}\in\mathbb{R}^s}  \left\{\Phi_p\big(M_F(\vec{w})\big):\ \vec{w}\geq\vec{0},\ \sum_{i\in[s]} c_i w_i \leq B\right\}
\label{Ppbdgcont}
\end{equation}
of Problem~\eqref{Ppbdg}. Then, for any design
$\vec{n}\in\mathbb{N}^s$, the following inequality holds:
$$\frac{1}{B} \sum_{i\in[s]} c_i n_i^p (w_i^*)^{1-p} \leq \frac{\varphi_p(\vec{n})}{\varphi_p(\vec{w^*})}.$$
\end{proposition}
\begin{proof}
First note that after the change of variable $z_i:= N B^{-1} c_i w_i$,
the continuous
relaxation~\eqref{Ppbdgcont} can be rewritten under the standard form~\eqref{Pcont},
where the matrix $M_i$ is replaced by \mbox{$M_i' = B (N c_i)^{-1} M_i$}.
Hence, we know from Proposition~\ref{prop:KKT_cond} that
the optimality conditions of Problem~\eqref{Ppbdgcont}
are:
$$\forall i \in [s],\quad Bc_i^{-1} \trace(M_F(\vec{w^*})^{p-1} M_i) \leq \varphi_p\big(\vec{w^*}\big),$$
with inequality if $w_i^*>0$. Then, we can apply exactly the same reasonning
as in the proof of Proposition~\ref{prop:boundW}, to show that
$$\forall S\subset[s],\quad 
\frac{1}{B} \sum_{i\in S} c_i (w_i^*)^{1-p} \leq \frac{\varphi_p(S)}{\varphi_p(\vec{w^*})}.$$
The only change is that the optimality conditions
must be multiplied by a factor proportional to $c_i (w_i^*)^{1-p}$
(instead of $(w_i^*)^{1-p}$ as in Equation~\eqref{proporKKT}).
Finally, we can apply the same arguments as in the proof of~\ref{prop:boundW_n}
to obtain the inequality of this proposition. 
\end{proof}

\subsection{Prior bounds}

In this section, we derive \emph{prior bounds}
for the solution obtained by rounding the continuous solution
of Problem~\eqref{Pcont}, i.e. approximation bounds which
depend only on the parameters $p$, $N$ and $s$ of
Problems~\eqref{Pp_intro} and~\eqref{Ppbin}.
We first need to state one technical lemma.

\begin{lemma} \label{lemma:minsummax}
Let $\vec{w} \in \mathbb{R}^s$ be a nonnegative vector summing to $r\leq s$,
$r\in \mathbb{N}$, and $p$ be an arbitrary real in the interval $[0,1]$. Assume without loss of generality
that the coordinates of $\vec{w}$ are sorted, i.e. $w_1 \geq \ldots \geq w_s \geq 0$.
If one of the following two conditions holds:
\begin{align*}
 (i)&\quad \forall i\in[s],\ w_i\leq 1 ; \\
 (ii)&\quad p \leq 1 - \frac{\ln r}{\ln s},
\end{align*}
then, the following inequality holds:
$$\frac{1}{r} \sum_{i=1}^r w_i^{1-p} \geq \left(\frac{r}{s}\right)^{1-p}.$$
\end{lemma}

\begin{proof}

We start by showing the lemma under the condition $(i)$. To this end, we consider the minimization problem
\begin{equation}
\min_{\vec{w}} \{\sum_{i=1}^r w_i^{1-p}:\ \sum_{i=1}^s w_i = r ;\ 1\geq w_1 \geq \ldots \geq w_s\geq 0\}. \label{minimaxcoord3} 
\end{equation}

Our first claim is that the optimum is necessarily attained
by a vector of the form
$\vec{w} = [ u + \alpha_1,\ldots,u+\alpha_r,u,\ldots,u]^T,$
where $\alpha_1,\ldots,\alpha_r\geq 0$, i.e.\ the $s-r$ coordinates
of $\vec{w}$ which are not involved in the objective function are equal.
To see this, assume \emph{ad absurbium} that $\vec{w}$ is optimal for Problem~\eqref{minimaxcoord3}, with $w_i>w_{i+1}$ for an index $i>r$.
Define $k$ as the smallest integer such that $w_1=w_2=\ldots=w_k>w_{k+1}$.
Then, $\vec{e_i} - 1/k \sum_{j\in[k]} \vec{e_j}$ is a feasible direction along
which the objective criterion $\sum_{i=1}^r w_i^{1-p}$ is decreasing,
a contradiction.
Problem~\eqref{minimaxcoord3} is hence equivalent to:

\begin{equation}
\min_{u,\vec{\alpha}} \{\sum_{i=1}^r (u+\alpha_i)^{1-p}:\ \sum_{i=1}^r \alpha_i = r-su ;\ 0\leq u;\ 0 \leq \alpha_i \leq 1-u\ (\forall i\in[r])  \}. \label{minimaxcoord4} 
\end{equation}
It is known that the objective criterion of Problem~\eqref{minimaxcoord4} is Schur-concave, as a
symmetric separable sum of concave functions (we refer
the reader to the book of Marshall and Olkin~\cite{MO79} for details about
the theory of majorization and Schur-concavity).
This tells us that for all $u\in[0,\frac{r}{s}]$,
the minimum with respect to $\vec{\alpha}$ is attained by
$$\vec{\alpha}=[\underbrace{1-u,\ldots,1-u}_{k\ \textrm{times}},r-su-k(1-u),0,\ldots,0]^T,$$
where $k=\lfloor \frac{r-su}{1-u} \rfloor$ (for a given $u$, this vector majorizes all the vectors of the feasible set). Problem~\eqref{minimaxcoord4}
can thus be reduced to the scalar minimization problem
$$
\min_{u\in[0,\frac{r}{s}]}\ \left\lfloor \frac{r-su}{1-u} \right\rfloor +
\Big( u + r -su - \left\lfloor \frac{r-su}{1-u} \right\rfloor (1-u) \Big)^{1-p}
+\big(r-\left\lfloor \frac{r-su}{1-u} \right\rfloor-1\big) u^{1-p}.
$$
It is not difficult to see that this function is piecewise concave,
on the $r-1$ intervals of the form $u\in\left[\frac{r-(k+1)}{s-(k+1)}, \frac{r-k}{s-k} \right],\ k\in[r-1]$, corresponding to the domains where
$k=\lfloor \frac{r-su}{1-u} \rfloor$ is constant. It follows that the minimum
is attained for a $u$ of the form $\frac{r-k}{s-k}$, where $k \in [r]$,
and the problem reduces to
$$\min_{k\in [r]}\ k + (r-k) \left(\frac{r-k}{s-k}\right)^{1-p}.$$
Finally, one can check that the objective function of the latter
problem is nondecreasing with respect to $k$, such that the minimum
is attained for $k=0$ (which corresponds to the uniform weight vector
$\vec{w}=[r/s,\ldots,r/s]^T$). This achieves the first part of this proof.\par
\vspace{5mm}

The proof of the lemma for the condition $(ii)$ is similar. This time,
we consider the minimization problem
\begin{equation}
\min_{\vec{w}} \{\sum_{i=1}^r w_i^{1-p}:\ \sum_{i=1}^s w_i = r ;\ w_1 \geq \ldots \geq w_s\geq 0\}. \label{minimaxcoord} 
\end{equation}
Again, the optimum is attained by a vector of the form $\vec{w} = [ u + \alpha_1,\ldots,u+\alpha_r,u,\ldots,u]^T,$
which reduces the problem to:
\begin{equation}
\min_{u,\vec{\alpha}} \{\sum_{i=1}^r (u+\alpha_i)^{1-p}:\ \sum_{i=1}^r \alpha_i = r-su ;\ u,\alpha_1,\ldots,\alpha_r\geq0 \}. \label{minimaxcoord2} 
\end{equation}
For a fixed $u$, the Schur-concavity of the objective function indicates
that the minimum is attained for $\vec{\alpha}=[r-su,0,\ldots,0]^T$. Finally, Problem~\eqref{minimaxcoord2} reduces to the scalar minimization problem
$$\min_{u\in[0,\frac{r}{s}]}\ \big( u+(r-su) \big)^{1-p} + (r-1) u^{1-p},$$
where the optimum is always attained for $u=0$ or $u=r/s$ by concavity.
It now is easy to see that the inequality of the lemma is satisfied when
the latter minimum is attained for $u=r/s$, i.e. if $r(\frac{r}{s})^{1-p} \leq r^{1-p}$,
which is equivalent to the condition $(ii)$ of the lemma.
\end{proof}

As a direct consequence of this lemma, we obtain a \emph{prior} approximation bound
for Problem~\eqref{Ppbin} when $p$ is in a neighborhood of $0$.
\begin{theorem}[Approximation bound for $N-$binary designs obtained by rounding]
\label{theo:factorFbin}
Let $p\in[0,1]$, $N\leq s$ and
$\vec{w^*}$ be a solution of the continuous optimal design problem~\eqref{Pcont}.
Let $S$ be the $N-$binary design obtained by selecting the $N$ largest
coordinates of $\vec{w^*}$. If $p\leq 1-\frac{\ln N}{\ln s}$, then
we have
$$\frac{\varphi_p(S)}{\varphi_p(S^*)} \geq \frac{\varphi_p(S)}{\varphi_p(\vec{w^*})} \geq \Big(\frac{N}{s}\Big)^{1-p}.$$
\end{theorem}

\begin{proof}
This is straightforward if we combine the result of Proposition~\ref{prop:boundW} and the one of Lemma~\ref{lemma:minsummax}
for $r=N$ and condition $(ii)$.
\end{proof}

In the next theorem, we give an approximation factor for the
design provided by Algorithm~\ref{algo:greedyrounding}.
This factor $F$ is plotted as a function of $p$ and the ratio $\frac{N}{s}$ on Figure~\ref{fig:F3d}.
For every value of $\frac{N}{s}$, this theorem shows that there is a continuously
increasing difficulty from the easy case ($p=1$, where $F=1$) to the most
degenerate problem ($p=0$, where $F=\min(\frac{N}{s},1-\frac{s}{4N})$).

\begin{theorem}[Approximation bound for $N-$replicated designs obtained by rounding]
\label{theo:factorF}
 Let $p\in[0,1]$,
$\vec{w^*}$ be a solution of the continuous optimal design problem~\eqref{Pcont}
and $\vec{n}$ be the vector returned
by Algorithm~\ref{algo:greedyrounding} for the input $\vec{w}=\vec{w^*}$.
Then, we have
$$\frac{\varphi_p(\vec{n})}{\varphi_p(\vec{n^*})} \geq \frac{\varphi_p(\vec{n})}{\varphi_p(\vec{w^*})} \geq F,$$
where $F$ is defined by:
$$F= \left\{\begin{array}{cll}
      \left(\frac{N}{s} \right)^{1-p} & \qquad \textrm{if } \left(\frac{N}{s} \right)^{1-p} \leq \frac{1}{2-p} 
      &\quad  (\textrm{in particular, if } \frac{N}{s}\leq e^{-1}); \\
     1-\frac{s}{N}(1-p)\left(\frac{1}{2-p}\right)^{\frac{2-p}{1-p}} & \qquad
    \textrm{Otherwise}
    & \quad (\textrm{in particular, if } \frac{N}{s} \geq \frac{1}{2});
     \end{array} \right.$$

\end{theorem}

\begin{figure}[ht]
\begin{center}
\begin{tabular}{cc}
 \includegraphics[width=0.5\textwidth]{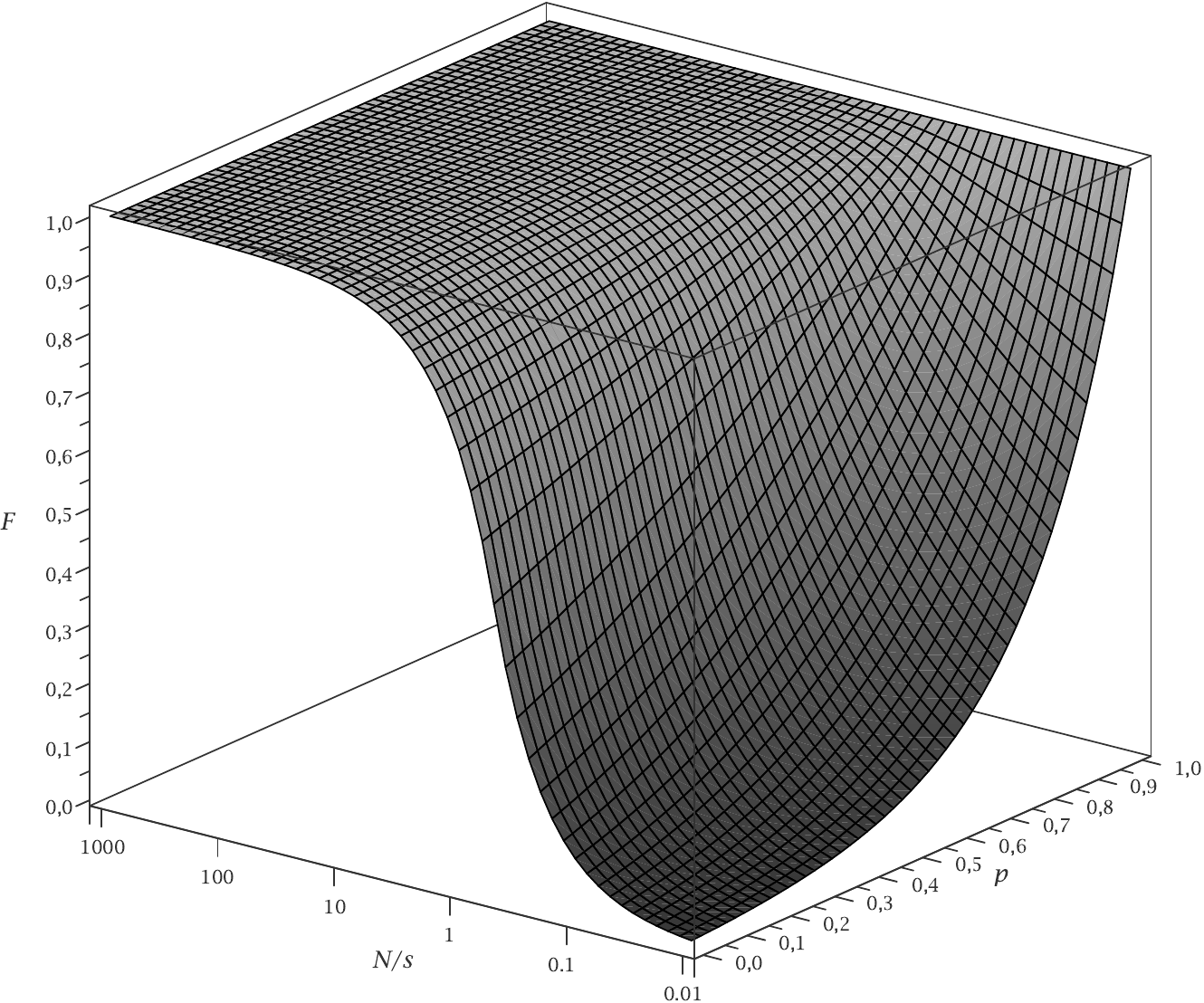} &
 \includegraphics[width=0.45\textwidth]{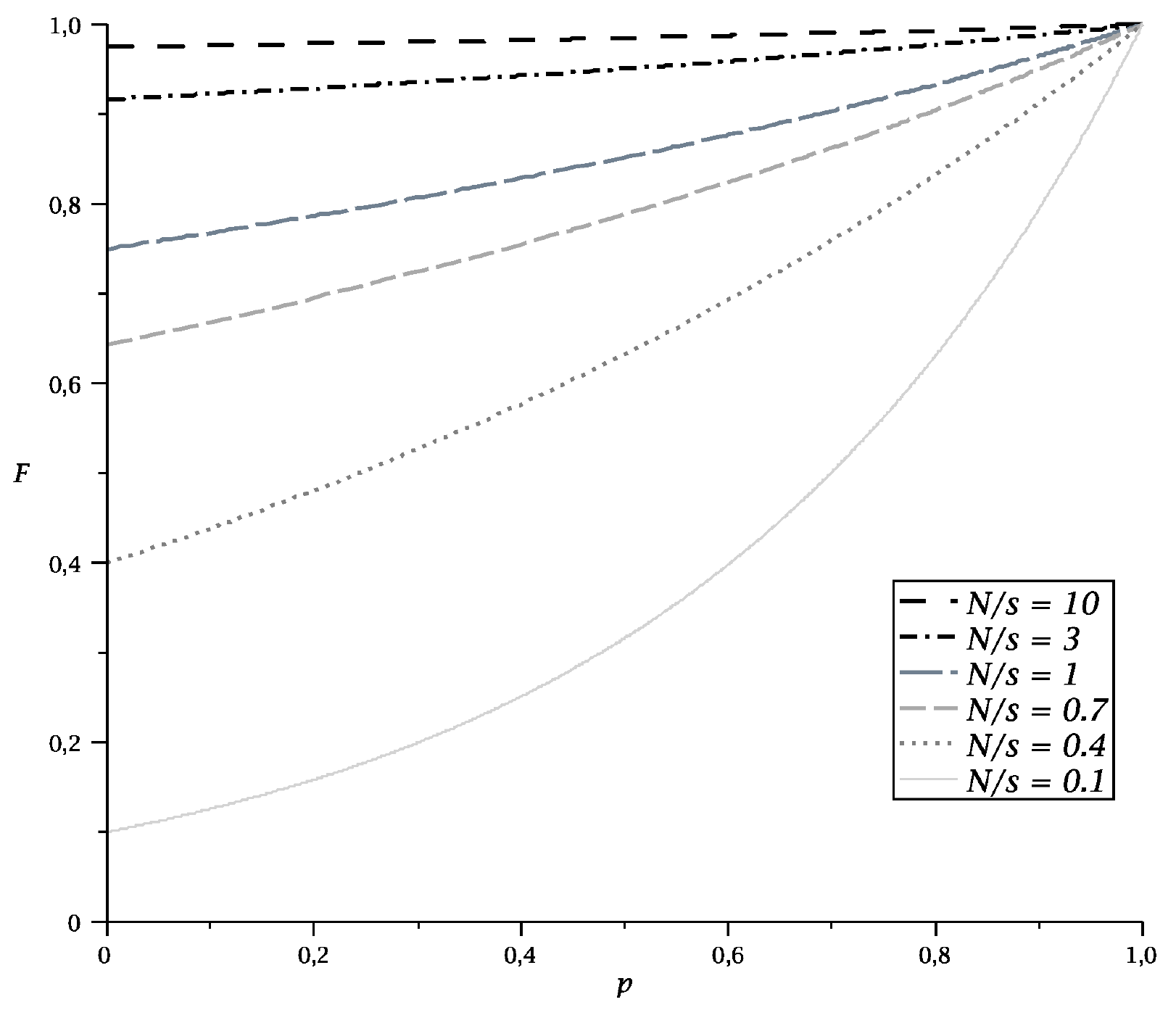}\\
(a) & (b)
\end{tabular}
\end{center}
\caption{Approximation factor F of Theorem~\ref{theo:factorF}: (a) as a function
of $p$ and the ratio $\frac{N}{s}$ (log scale); (b) as a function
of $p$ for selected values of $\frac{N}{s}$. \label{fig:F3d}}
\end{figure}

\begin{proof}
For all $i\in [s]$ we denote by $f_i:=w_i^*-\lfloor w_i^* \rfloor$ the fractional
part of $w_i^*$, and we assume without loss of generality that these numbers
are sorted, i.e.\ , $f_1\geq f_2 \geq \ldots \geq f_s$.
We will prove the theorem through a simple (suboptimal) rounding $\vec{\bar{n}}$,
which we define as follows: 
$$\bar{n}_i = \left\{ \begin{array}{cl}
                       \lfloor w_i^* \rfloor +1 & \quad \textrm{if } i \leq N - \sum_{i \in [s]} \lfloor w_i^* \rfloor; \\
		       \lfloor w_i^* \rfloor & \quad \textrm{Otherwise.}
                      \end{array} \right.
$$
We know from Proposition~\ref{prop:boundW_n} and from the fact that
Algorithm~\ref{algo:greedyrounding} solves Problem~\eqref{probaposteriori}
the integer vector $\vec{n}$
satisfies
\begin{equation}
N \frac{\varphi_p(\vec{n})}{\varphi_p(\vec{w^*})} \geq
\sum_{i=1}^s n_i^p (w_i^*)^{1-p} \geq \sum_{i=1}^s \bar{n}_i^p (w_i^*)^{1-p}.
\label{ineqnbar}
\end{equation}
We shall now bound from below the latter expression:
if $\bar{n}_i =\lfloor w_i^* \rfloor$ and $\lfloor w_i^* \rfloor \neq 0$, then 
\begin{equation}
\bar{n}_i^p (w_i^*)^{1-p} = \lfloor w_i^* \rfloor \left( \frac{w_i^*}{\lfloor w_i^* \rfloor} \right)^{1-p} \geq \lfloor w_i^* \rfloor. \label{ineqicase1} 
\end{equation}
Note that Inequality~\eqref{ineqicase1} also holds if $\bar{n}_i =\lfloor w_i^* \rfloor=0$.
If $\bar{n}_i =\lfloor w_i^* \rfloor + 1$, we write
\begin{equation}
\bar{n}_i^p (w_i^*)^{1-p} = \underbrace{\left(\frac{w_i^*}{\bar{n}_i}\right)^{1-p} + \ldots +  \left(\frac{w_i^*}{\bar{n}_i}\right)^{1-p}}_{\bar{n}_i\textrm{  terms}}
\geq \underbrace{1^{1-p} + \ldots +1^{1-p}}_{
\lfloor w_i^* \rfloor\textrm{ terms}} + f_i^{1-p}
=\lfloor w_i^* \rfloor + f_i^{1-p},
\label{ineqicase2}
\end{equation}
where the inequality is a consequence of the concavity of
$\vec{w} \mapsto \sum_j w_j^{1-p}$. Combining Inequalities~\eqref{ineqicase1}
and~\eqref{ineqicase2} yields
$$\sum_{i=1}^s \bar{n}_i^p (w_i^*)^{1-p} \geq \sum_{i=1}^s \lfloor w_i^* \rfloor
+ \sum_{j=1}^{N-\sum_{i=1}^s \lfloor w_i^* \rfloor} f_i^{1-p}= \bar{N} + 
\sum_{j=1}^{N-\bar{N}} f_i^{1-p},$$
where we have set $\bar{N}:=\sum_{i=1}^s \lfloor w_i^* \rfloor \in \{\max(N-s+1,0),\ldots,N\}$.
Since the vector $\vec{f}=[f_1,\ldots,f_s]$ sums to $N-\bar{N}$,
we can apply the result of Lemma~\ref{lemma:minsummax} with condition $(i)$,
with $r=N-\bar{N}$, and we obtain
$$\sum_{i=1}^s \bar{n}_i^p w_i^{1-p} \geq \bar{N} + (N-\bar{N}) \left( \frac{N-\bar{N}}{s} \right)^{1-p} \geq \min_{u\in [0,N]} u + (N-u) \left( \frac{N-u}{s} \right)^{1-p}.$$

We will compute this lower bound in closed-form, which will provide the approximation
bound of the theorem. To do this, we define the function
$g: u \mapsto u + (N-u) \left( \frac{N-u}{s} \right)^{1-p}$
on $]-\infty,N]$, and we observe (by differentiating) that
$g$ is decreasing on $]-\infty,u^*]$ and increasing on $[u^*,N[$, where
$$u^* = N - s \left(\frac{1}{2-p}\right)^{\frac{1}{1-p}}.$$
Hence, only two cases can appear: either $u^*\leq 0$, and the
minimum of $g$ over $[0,N]$ is attained for $u=0$; or $u^*\geq 0$,
and $g_{|[0,N]}$ attains its minimum at $u=u^*$. Finally, the bound given in
this theorem is either $N^{-1} g(0)$ or $N^{-1} g(u^*)$,
depending on the sign of $u^*$.
In particular, since the function $$h: p \mapsto \left(\frac{1}{2-p}\right)^{\frac{1}{1-p}}$$
is nonincreasing
on the interval $[0,1]$, with $h(0)=\frac{1}{2}$ and $h(1)=e^{-1}$,
we have:
$$\forall p\in[0,1], \quad \frac{N}{s}\leq e^{-1} \Longrightarrow u^* \leq 0
\quad\textrm{ and }\quad  \frac{N}{s}\geq \frac{1}{2} \Longrightarrow u^* \geq 0.$$
\end{proof}

\begin{remark}
The alternative rounding $\vec{\tilde{n}}$ is very useful to obtain the
formula of Theorem~\ref{theo:factorF}. However, since $\vec{\tilde{n}}$
differs from the design $\vec{n}$ returned by Algorithm~\ref{algo:greedyrounding}
in general, the inequality $\frac{\varphi_p(\vec{n})}{\varphi_p(\vec{w^*})} \geq F$ is not tight. Consider for example
the situation where $p=0$ and $N=s$, which is a trivial case
for the rank optimization problem~\eqref{maxrank_intro}:
the incremental rounding algorithm always returns a design $\vec{n}$
such that $(w_i^*>0 \Rightarrow n_i>0)$, and hence the problem
is solve to optimality (the design is of full rank). In contrast, Theorem~\ref{theo:factorF}
only guarantees a factor $F=\frac{3}{4}$ for this class of instances.
\end{remark}

\begin{remark}
We point out that Theorem~\ref{theo:factorF} improves on the
greedy approximation factor $1-e^{-1}$ in many situations. The gray area 
of Figure~\ref{fig:bettergreedy} shows the values of $(\frac{N}{s},p) \in \RR^*_+ \times [0,1]$
for which the approximation guarantee is better with Algorithm~\ref{algo:greedyrounding}
than with the greedy algorithm of section~\ref{sec:submod}.
\end{remark}

\begin{figure}
 \begin{center}
  \includegraphics[width=0.6\textwidth]{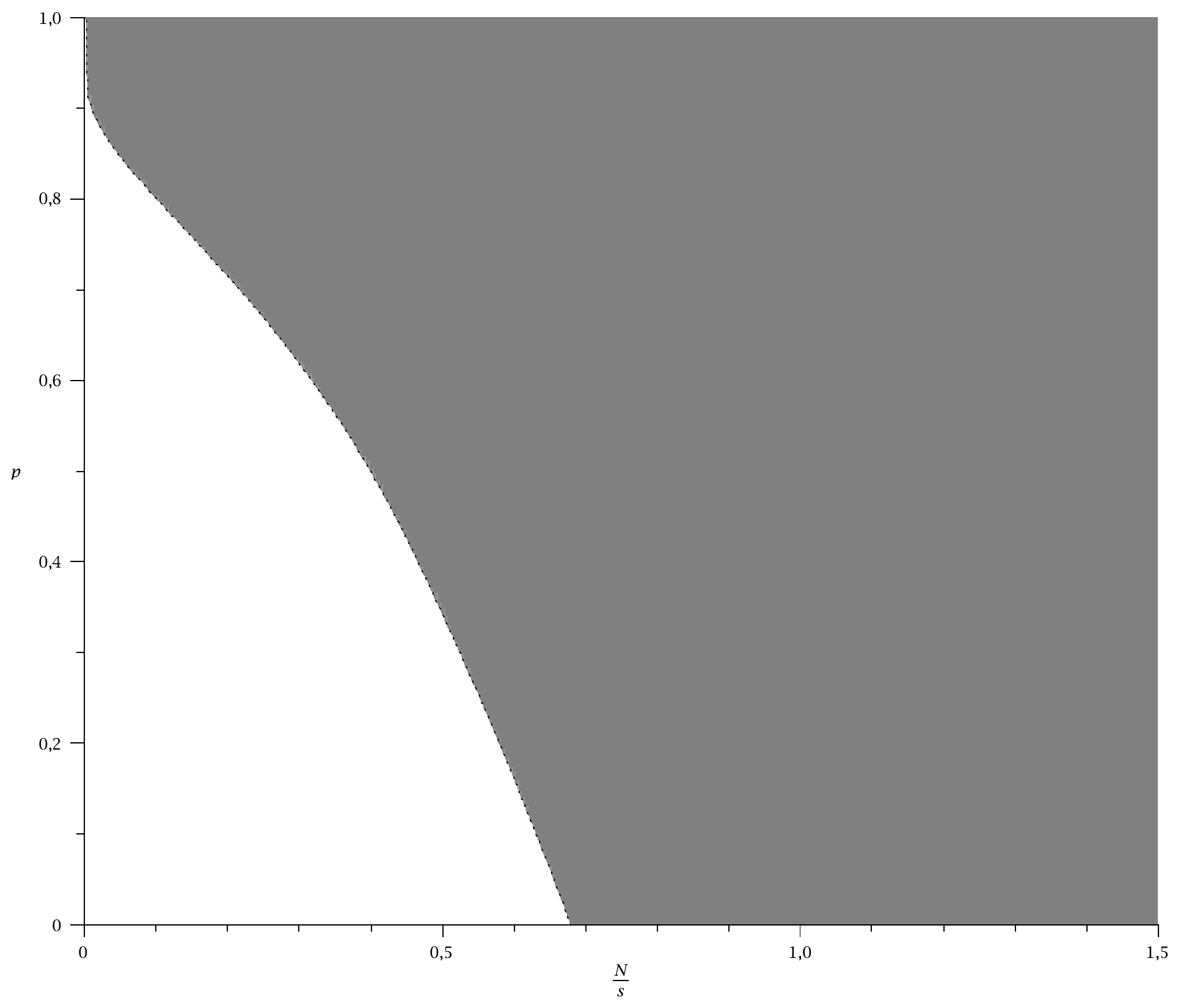}
 \end{center}
\caption{in gray, values of $(\frac{N}{s},p) \in \RR^*_+ \times [0,1]$ such that the factor $F$
of Theorem~\ref{theo:factorF} is larger than $1-e^{-1}$. \label{fig:bettergreedy}}
\end{figure}

\begin{remark}
Recall that the relevant criterion for the theory of optimal design
is the \emph{positively homogeneous} function $\vec{w} \mapsto \Phi_p\big(M_F(\vec{w})\big) = m^{-1/p} \varphi_p(\vec{w})^{1/p}$ (cf.\ Equation~\eqref{phip}).
Hence, if a design is within a factor $F$ of the optimum with respect to $\varphi_p$,
its $\Phi_p-$efficiency is $F^{1/p}$. 
In the \emph{overinstrumented} case $N>s$, Pukelsheim gives a rounding procedure
with a $\Phi_p-$efficiency of $1-\frac{s}{N}$ (Chapter~12 in~\cite{Puk93}).
We have plotted in Figure~\ref{fig:betterpuk} the area of the domain
$(\frac{s}{N},p)\in[0,1]^2$ where the approximation guarantee of Theorem~\ref{theo:factorF} is better.
\end{remark}

\begin{figure}
 \begin{center}
  \includegraphics[width=0.6\textwidth]{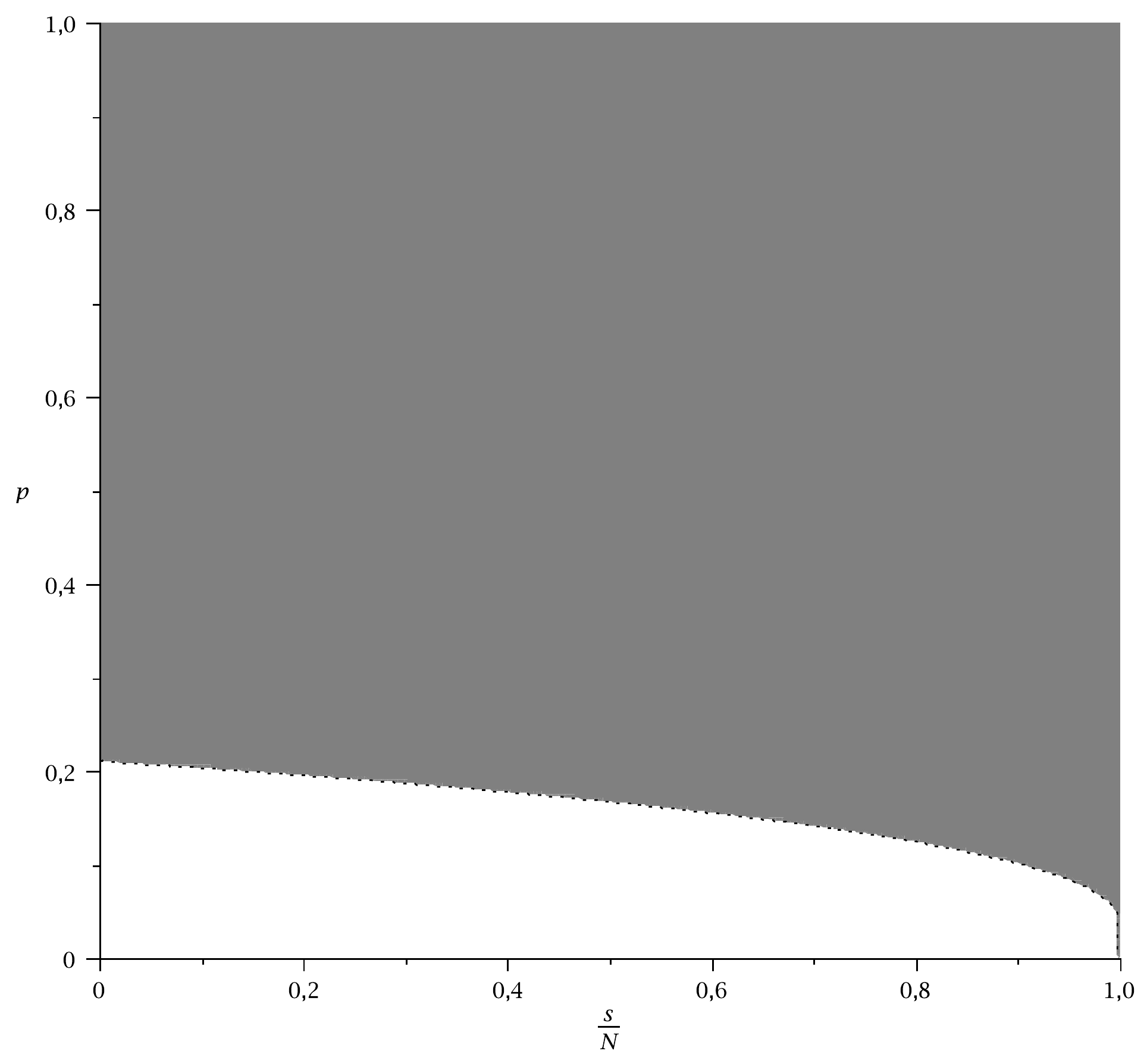}
 \end{center}
\caption{in gray, values of $(\frac{s}{N},p)\in[0,1]^2$ such that
the factor $F$ of Theorem~\ref{theo:factorF} is larger than $(1-s/N)^p$. \label{fig:betterpuk}}
\end{figure}

\section{Conclusion}

This paper gives bounds on the behavior of some classical heuristics used for combinatorial problems arising in optimal experimental design.
Our results can either justify or discard the use of such heuristics, depending on the settings of the instances considered.
Moreover, our results confirm some facts that had been observed in the literature, namely 
that rounding algorithms perform better if the density of measurements is high, and that
the greedy algorithm always gives a quite good solution. We illustrate these observations with two examples:

In a sensor location problem, Uci\'{n}ski and Patan~\cite{UP07} noticed that the trimming of a Branch and Bound
algorithm was better if they activated more sensors,
although this led to a much larger search space. The authors claims that this surprising result can be explained
by the fact that a higher density of sensors leads to a better continuous relaxation. 
This is confirmed by our result of approximability,
which shows that the larger is the number of selected experiments, the better is the quality of the rounding.

It is also known
that the greedy algorithm generally gives very good results for the optimal design of experiments
(see e.g.~\cite{SQZ06}, where the authors
explicitly chose not to implement a local search from the design greedily chosen, since the greedy algorithm already performs very well).
Our $(1-1/e)-$approximability
result guarantees that this algorithm always well behaves indeed.

\section{Acknowledgment}

The author wishes to thank St\'ephane Gaubert for his useful comments and advises,
for suggesting to work within the framework of matrix inequalities and submodularity,
and of course for his warm support. He also wants to thank Mustapha Bouhtou for the
stimulating discussions which are at the origin of this work.
The author expresses his gratitude to an anonymous referee of the ISCO conference --
where an announcement of some of the present
results was made~\cite{BGSagnol10ENDM}, and to three referees of DAM for precious comments and suggestions.


\newpage
\renewcommand \appendixname{}
\appendix
\noindent{\textbf {\Large Appendix}}

\section{From optimal design of statistical experiments to Problem~\texorpdfstring{\eqref{Pp_intro}}{(Pp)}} \label{sec:problemStatement}

\subsection{The classical linear model}

We denote vectors by bold-face lowercase letters and we make use of the classical
notation $[s]:=\{1,\ldots,s\}$ (and we define $[0]:=\emptyset$). The set of nonnegative (resp.\ positive) real numbers is
denoted by $\RR_+$ (resp.\ $\RR_+^*$), the set of $m\times m$ symmetric
(resp.\ symmetric positive semidefinite, symmetric positive definite) is denoted
by $\mathbb{S}_m$ (resp.\ $\mathbb{S}_m^+$, $\mathbb{S}_m^{++}$). The expected value of a random variable $X$ is denoted by $\EE[X]$.

We denote by $\vec{\theta} \in \RR^m$ the vector of the parameters that we want to estimate. In accordance with the
classical linear model, we assume that the experimenter has a collection
of $s$ experiments at his disposal,
each one providing a (multidimensional) observation which is a linear combination of the parameters,
up to a noise on the measurement whose covariance matrix is known and positive definite.
In other words, for each experiment $i \in [s]$, we have 
\begin{equation}
 \vec{y_i}=A_{i} \vec{\theta} + \vec{\epsilon_i},\qquad \EE[\vec{\epsilon_i}]=\vec{0}, \qquad \EE[\vec{\epsilon_i}\vec{\epsilon_i}^T]=\Sigma_i, \label{mesEq}
\end{equation}
where $\vec{y_i}$ is the vector of measurement of size $l_i$, $A_{i}$ is a $(l_i \times m)-$matrix, and $\Sigma_i\in \mathbb{S}_{l_i}^{++}$ is a known covariance matrix.
We will assume that the noises have unit variance 
for the sake of simplicity: $\Sigma_i=I$. We may always reduce to this case by a left 
multiplication of the observation equation~\eqref{mesEq} by $\Sigma_i^{-1/2}$. The errors on the measurements are assumed to be mutually independent, i.e.
$$ i \neq j \Longrightarrow \EE[\vec{\epsilon_i} \vec{\epsilon_j}^T]=0.$$

As explained in the introduction, the aim of experimental design theory is to
choose how many times each experiment will be performed
so as to maximize the accuracy of the estimation of $\vec{\theta}$,
with the constraint that
$N$ experiments may be conducted. We therefore define the integer-valued \textit{design}
variable $\vec{n}\in\mathbb{N}^s$, where $n_k$ indicates how many times the experiment $k$
is performed. We denote by $i_k \in [s]$ the index of the $k\th$ conducted experiment (the order
in which we consider the measurements has no importance),
so that the aggregated vector of observation reads:
\begin{equation}
 \vec{y}=\mathcal{A}\ \vec{\theta} + \vec{\epsilon},
\end{equation}
\begin{equation*}
\textrm{ where   }  \vec{y}=\left[
\begin{array}{c}
\vec{y_{i_1}} \\ 
\vdots \\ 
\vec{y_{i_N}}
\end{array} \right],\qquad  \mathcal{A}=
\left[
\begin{array}{c}
A_{i_1} \\ 
\vdots \\ 
A_{i_N}
\end{array}
\right],\qquad  \EE[\vec{\epsilon}]=\vec{0},\quad \mathrm{and}\quad \EE[\vec{\epsilon} \vec{\epsilon}^T]=I.
\end{equation*}

Now, assume that we have enough measurements, so that $\mathcal{A}$ is of full rank. A common result in the field of statistics, known as the \textit{Gauss-Markov} theorem, states that the best linear unbiased estimator of $\vec{\theta}$ is given by a pseudo inverse formula. Its variance is given below:
\begin{align}
\hat{\vec{\theta}}  =\Big(\mathcal{A}^T \mathcal{A} \Big)^{-1} \mathcal{A}^T \vec{y}. \label{bestEstimator} \\
\textrm{Var}(\hat{\vec{\theta}})  =(\mathcal{A}^T \mathcal{A} )^{-1}. \label{bestVar}
\end{align}

We denote the inverse of the covariance matrix~\eqref{bestVar} by $M_F(\vec{n})$,
because in the Gaussian case it coincides with the Fisher information matrix
of the measurements.
Note that it can be decomposed as the sum of the information
matrices of the selected experiments:
\begin{eqnarray}
M_F(\vec{n}) &=& \mathcal{A}^T \mathcal{A} \nonumber \\
&=&\sum_{k=1}^N A_{i_k}^T A_{i_k} \nonumber \\
&=&\sum_{i=1}^s n_i A_{i}^T A_{i}. \label{fisher} 
\end{eqnarray}
The classical experimental design approach consists in choosing the design $\vec{n}$
in order to make the variance of the estimator~(\ref{bestEstimator}) \textit{as small as possible}.
The interpretation is straightforward: with the assumption that the noise $\vec{\epsilon}$
is normally distributed, for every probability level $\alpha$,
the estimator $\hat{\vec{\theta}}$ lies in the confidence ellipsoid centered
at $\vec{\theta}$ and defined by the following inequality:
\begin{align}
\label{e-confidence}
(\vec{\theta}-\hat{\vec{\theta}})^T Q(\vec{\theta}-\hat{\vec{\theta}}) \leq \kappa_\alpha,
\end{align}
where $\kappa_\alpha$ depends on the specified probability level, and $Q=M_F(\vec{n})$
is the inverse of the covariance matrix $\textrm{Var}(\hat{\vec{\theta}})$.
We would like to make these confidence ellipsoids \textit{as small as possible},
in order to reduce the uncertainty on the estimation of $\vec{\theta}$.
To this end, we can express the inclusion of ellipsoids in terms of matrix inequalities.
The space of symmetric matrices is equipped with the {\em L\"owner ordering}, which is defined by
$$ \forall B,C \in \mathbb{S}_m, \qquad B \succeq C \Longleftrightarrow B-C \in \mathbb{S}_m^+. $$
Let $\vec{n}$ and $\vec{n'}$ denote two designs such that the matrices $M_F(\vec{n})$
and $M_F(\vec{n'})$ are invertible. 
One can readily check that for any value of the probability level $\alpha$,
the confidence ellipsoid~\eqref{e-confidence}
corresponding to $Q=M_F(\vec{n})$ is included in the confidence ellipsoid corresponding
to $Q'=M_F(\vec{n'})$ if and only if $M_F(\vec{n})\succeq M_F(\vec{n'})$.
Hence, we will prefer the design $\vec{n}$ to the design $\vec{n'}$ if the latter inequality
is satisfied.

\subsection{Statement of the optimization problem}

Since the L\"owner ordering on symmetric matrices is only a partial ordering, the problem consisting in maximizing $M_F(\vec{n})$ is ill-posed.
So we will rather maximize a scalar \textit{information function} of the Fisher matrix, i.e.\ a function mapping $\mathbb{S}_m^{+}$ onto
the real line, and which satisfies natural properties, such as positive homogeneity, monotonicity with respect to L\"owner ordering,
and concavity. For a more detailed description of the information functions, the reader is referred to the book of Pukelsheim~\cite{Puk93}, 
who makes use of the class of matrix means $\Phi_p$, as first proposed by Kiefer~\cite{Kief75}.
These functions are defined like the $L_p$-norm
of the vector of eigenvalues of the Fisher information matrix, but for $p \in [-\infty,1]$:
for a symmetric positive definite matrix $M\in\mathbb{S}_m^{++}$, $\Phi_p$ is defined by
 \begin{equation}
 \Phi_p(M)=\left\{ 
\begin{array}{ll}
\lambda_{\mathrm{min}}(M) & \textrm{for $p=-\infty$ ;} \\ 
(\frac{1}{m}\ \mathrm{trace}\ M^p)^{\frac{1}{p}} & \textrm{for $p \in\ ]-\infty,1],\ p \neq 0$;}  \\ 
(\operatorname{det}(M))^{\frac{1}{m}} & \textrm{for $p=0$,}
\end{array}
\right. \label{phip}
\end{equation}

\noindent where we have used the extended definition of powers of matrices $M^p$ for arbitrary real parameters
$p$: if $\lambda_1,\ldots,\lambda_m$ are the eigenvalues of $M$
counted with multiplicities, $\mathrm{trace}\ M^p=\sum_{j=1}^m \lambda_j^p$.
For singular positive semi-definite matrices $M \in \mathbb{S}_m^{+}$, $\Phi_p$ is defined by continuity:

\begin{equation}
 \Phi_p(M)=\left\{ 
\begin{array}{ll}
0 & \textrm{for $p \in [-\infty,0]$ ;}  \\ 
(\frac{1}{m}\ \mathrm{trace}\ M^p)^{\frac{1}{p}} & \textrm{for $p \in\ ]0,1]$. }
\end{array}
\right. \label{phising}
\end{equation}

The class of functions $\Phi_p$
includes as special cases the classical optimality criteria used in the experimental
design literature, namely $E-$optimality for $p=-\infty$
(smallest eigenvalue of $M_F(\vec{n})$),
 $D-$optimality for $p = 0$ (determinant
of the information matrix), $A-$optimality for $p=-1$
(harmonic average of the eigenvalues), and $T-$optimality for $p=1$ (trace).
The case $p=0$ (D-optimal design) admits a simple geometric interpretation:
the volume of the confidence ellipsoid~\eqref{e-confidence}
is given by $C_m\kappa_\alpha^{m/2}\det(Q)^{-1/2}$ where
$C_m>0$ is a constant depending only on the dimension.
Hence, maximizing $\Phi_0(M_F(\vec{n}))$
is the same as minimizing the volume of every confidence ellipsoid.

We can finally give a mathematical formulation to the problem of selecting
$N$ experiments to conduct among the set $[s]$:

\begin{align}\label{optdesProblem}
 \max_{n_i \in \mathbb{N}\ (i=1,\ldots,s)} &\quad \Phi_p \Big( \sum_{i=1}^s n_i A_i^T A_i \Big) \\
 \operatorname{s.t.} &\quad \sum_i n_i \leq N, \nonumber
\end{align}

\subsection{The underinstrumented situation}

We note that the problem of maximizing the information matrix $M_F(\vec{n})$ with respect to the L\"owner ordering remains
meaningful even when $M_F(\vec{n})$ is not of full rank (the interpretation of $M_F(\vec{n})$
as \textit{the inverse of the covariance matrix of the best linear unbiased estimator} vanishes,
but $M_F(\vec{n})$ is still the Fisher information matrix of the experiments if the measurement errors are Gaussian).
This case does arise in \emph{underinstrumented situations}, in which some constraints may not
allow one to conduct a number of experiments which
is sufficient to infer all the parameters.

An interesting and natural idea to find an optimal under-instrumented design is to choose the design which maximizes
the rank of the observation matrix $\mathcal{A}$, or equivalently
of $M_F(\vec{n})=\mathcal{A}^T \mathcal{A}$.
The \textit{rank maximization} is a nice combinatorial problem, where we are looking for a
subset of matrices whose sum is of maximal rank:
\begin{align*} 
 \max_{\vec{n} \in \mathbb{N}^s} &\quad \rank \Big( \sum_i n_i A_{i}^T A_{i} \Big) \\
  \operatorname{s.t.}\ & \qquad \sum_i n_i \leq N. \nonumber
\end{align*}

When every feasible information matrix is singular, Equation~\eqref{phising} indicates that the maximization of
$\Phi_p(M_F(\vec{n}))$ can be considered only for nonnegative values of $p$. Then,
the next proposition shows that
$\Phi_p$ can be seen as a deformation of the rank criterion for $p \in ]0,1]$.
First notice that when $p>0$, the maximization of $\Phi_p(M_F(\vec{n}))$ is equivalent to:
\begin{align*} 
\displaystyle{\max_{\vec{n} \in \mathbb{N}^s}} & \quad
    \varphi_p\big(\vec{n}\big):=\ \textrm{trace}\: \Big(\displaystyle{\sum_i} n_i A_i^T A_i \Big)^p \\
 \st & \quad \quad\displaystyle{\sum_i}\ n_i \leq N. \nonumber
\end{align*}
If we set $M_i=A_i^T A_i$, we obtain the problems~\eqref{maxrank_intro} and~\eqref{Pp_intro}
which were presented in the first lines of this article.

\begin{proposition} \label{limp0} For all positive semidefinite matrix $M \in \mathbb{S}_m^+,$ 
\begin{equation} 
\lim_{p\rightarrow0^+} \mathrm{trace}\ M^p = \mathrm{rank}\ M.
\end{equation}
\end{proposition}

\begin{proof}
Let $\lambda_1,\ldots, \lambda_r$ denote the positive eigenvalues of $M$, counted with multiplicities, so that $r$ is the rank of $M$. We have the first order expansion as $p \to 0^+$:
\begin{align}
 \mathrm{trace}\ M^p = \sum_{k=1}^r \lambda_k^p  = r + p\ \log (\prod_{k=1}^r \lambda_k) + \mathcal{O}(p^2) \label{expansion1}
\end{align}
\end{proof} 
Consequently, $\mathrm{trace}\ M^0$ will stand for $\mathrm{rank}(M)$ in the sequel and the
rank maximization problem~\eqref{maxrank_intro}
is the limit of problem~\eqref{Pp_intro} as $p \to 0^+$.

\begin{corollary} \label{corop0}
If $p>0$ is small enough, then every design $\vec{\vec{n^*}}$ which is a solution
of Problem~\eqref{Pp_intro} maximizes the rank of $M_F(\vec{n})$. Moreover, among the designs
which maximize this rank, $\vec{\vec{n^*}}$ maximizes the product of nonzero eigenvalues
of $M_F(\vec{n})$.
\end{corollary}

\begin{proof} 
Since there is only a finite number of designs, it follows from~\eqref{expansion1} that for $p>0$ small enough, every design which maximizes $\varphi_p$
must maximize in the lexicographical order first the rank of $M_F(\vec{n})$,
and then the pseudo-determinant $\prod_{\{k:\lambda_k>0\}} \lambda_k$.
\end{proof}

\section{Proof of Proposition~\ref{prop:boundW}} \label{sec:proofIneq}

The proof of Proposition~\ref{prop:boundW} relies on several lemmas on the directional derivative of
a scalar function applied to a symmetric matrix, which we state next.
First recall that if $f$ is differentiable on $\RR_+^*$, then $f$ is Fr\'echet differentiable
over $\mathbb{S}_m^{++}$,
and for $M\in\mathbb{S}_m^{++}$, $H\in\mathbb{S}_m$, we denote by $D\!f(M)(H)$
its directional derivative at $M$ in the direction of $H$ (see Equation~\eqref{Dfrechet}).

\begin{lemma} \label{lemma:permDeriv}
 If $f$ is continuously differentiable on $\Rplusstar$, i.e.\ $f\in\mathcal{C}^1(\Rplusstar)$,
$M\in\Splusplus, A,B \in \mathbb{S}_m$, then
$$\trace(A\ D\!f(M)(B))=\trace(B\ D\!f(M)(A)).$$
\end{lemma}
\begin{proof}
Let $M=Q D Q^T$ be an eigenvalue decomposition of $M$. It is known (see e.g.~\cite{Bha97})
that $D\!f(M)(H)$
can be expressed as $Q (f^{[1]}(D) \odot Q^T H Q) Q^T$, where $f^{[1]}(D)$ is a symmetric
matrix called the \emph{first divided difference} of $f$ at $D$
and $\odot$ denotes the Hadamard (elementwise)
product of matrices. With little work, the latter derivative may be rewritten as:
$$D\!f(M)(H)=\sum_{i,j} f^{[1]}_{ij} \vec{q_i}\vec{q_i}^T H \vec{q_j}\vec{q_j}^T,$$
where $\vec{q_k}$ is the $k\th$ eigenvector of $M$ (i.e., the $k\th$ column of $Q$)
and $f^{[1]}_{ij}$ denotes the $(i,j)-$element of $f^{[1]}(D)$. We can now conclude:
\begin{align*}
\trace(A\ D\!f(M)(B)) & = \sum_{i,j} f^{[1]}_{ij} \trace ( A\vec{q_i}\vec{q_i}^T B \vec{q_j}\vec{q_j}^T) \\
  & = \sum_{i,j} f^{[1]}_{ji} \trace ( B\vec{q_j}\vec{q_j}^T H \vec{q_i}\vec{q_i}^T) \\
  & = \trace(B\ D\!f(M)(A))
\end{align*}
\end{proof}

We next show that when $f$ is antitone, the mapping $X\mapsto D\!f(M)(X)$ is nonincreasing
with respect to the L\"owner ordering.
\begin{lemma} \label{lemma:nonincDeriv}
 If $f$ is differentiable and antitone on $\Rplusstar$, then for all $A,B$ in $\mathbb{S}_m$,
$$A\preceq B \Longrightarrow D\!f(M)(A)\succeq D\!f(M)(B).$$
\end{lemma}
\begin{proof}
The lemma trivially follows from the definition of the directional derivative:
$$ D\!f(M)(A) = \lim_{\epsilon \to 0^+} \frac{1}{\epsilon} \big(f(M+\epsilon A) - f(M) \big)$$
and the fact that $A\preceq B$ implies $M+\epsilon A\preceq M+\epsilon B$ for all $\epsilon>0$.
\end{proof}

\begin{lemma} \label{lemma:commDeriv}
 Let $f$ be differentiable on $\Rplusstar$, $M\in \Splusplus$, $A \in \mathbb{S}_m$. If $A$ and $M$
commute, then
$$D\!f(M)(A)=f'(M)A \in \mathbb{S}_m,$$
where $f'$ denotes the (scalar) derivative of $f$.
\end{lemma}
\begin{proof}
 Since $A$ and $M$ commute, we can diagonalize them simultaneously:
$$M=Q \Diag(\vec{\lambda}) Q^T,\quad A=Q \Diag(\vec{\mu}) Q^T.$$
Thus, it is clear from the definition of the directional derivative that
$$D\!f(M)(A)=Q\ D\!f\big(\Diag(\vec{\lambda})\big)\big(\Diag(\vec{\mu})\big)\ Q^T.$$
By reasoning entry-wise on the diagonal matrices, we find:
$$D\!f\big(\Diag(\vec{\lambda})\big)\big(\Diag(\vec{\mu})\big)=
\Diag\big(f'(\lambda_1) \mu_1,\ldots,f'(\lambda_m) \mu_m\big)
=\Diag\big(f'(\vec{\lambda})\big)\Diag(\vec{\mu})$$
The equality of the lemma is finally obtained by writing:
$$D\!f(M)(A)=Q \Diag\big(f'(\vec{\lambda})\big)\Diag(\vec{\mu}) Q^T=
Q \Diag\big(f'(\vec{\lambda})\big)Q^T Q\Diag(\vec{\mu}) Q^T=f'(M) A.$$
Note that the matrix $f'(M) A$ is indeed symmetric, because $f'(M)$ and $A$ commute.
\end{proof}

Before we give the proof of the main result, we recall an important result from the theory of
optimal experimental designs, which characterizes the optimum of Problem~\eqref{Pcont}.
\begin{proposition}[General equivalence theorem~\cite{Kief74}] \label{prop:KKT_cond}
Let $p \in [0,1]$. A design $\vec{w^*}$ is optimal for Problem~\eqref{Pcont} if and only if:
$$
\forall i \in [s],\quad N \trace(M_F(\vec{w^*})^{p-1} M_i) \leq \varphi_p\big(\vec{w^*}\big).
$$
Moreover, the latter inequalities become equalities for all $i$
such that $w_i^*>0$.
\end{proposition}

For a proof of this result, see~\cite{Kief74} or Paragraph~7.19 in~\cite{Puk93}, where the problem is studied with the normalized
constraint $\sum_i w_i\leq 1$. In fact, the \emph{general equivalence theorem} details the Karush-Kuhn-Tucker
conditions of optimality of Problem~\eqref{Pcont}. To derive them, one can use the fact that
when $M_F(\vec{w})$ is invertible,
$$\frac{\partial \varphi_p(\vec{w})}{\partial w_i}=\trace(M_F(\vec{w})^{p-1} M_i)\quad \textrm{for all}\quad p\in]0,1],$$
and 
$$\frac{\partial \log\det(M_F(\vec{w}))}{\partial w_i}=\trace(M_F(\vec{w})^{-1} M_i).$$
Note that for $p\neq 1$, the proposition implicitly implies that $M_F(\vec{w^*})$ is invertible.
A proof of this fact can be found in Paragraph~7.13 of~\cite{Puk93}.

We can finally prove the main result:

\begin{proof}[Proof of Proposition~\ref{prop:boundW}]
Let $\vec{w^*}$ be an optimal solution to Problem~\eqref{Pcont} and $S$ be a subset of~$[s]$
such that $w_i^*>0$ for all $i \in S$
(the case in which $w_i^*=0$ for some index $i\in S$ will trivially follow
if we adopt the convention $0^0=0$).
We know from Proposition~\ref{prop:KKT_cond} that
$N^{-1} \varphi_p\big(\vec{w^*}\big) = \trace(M_F(\vec{w^*})^{p-1} M_i)$ for all $i$ in $S$.
If we combine these equalities by multiplying each expression
by a factor proportional to $(w_i^*)^{1-p}$, we obtain:
\begin{align}
\frac{1}{N} \varphi_p\big(\vec{w^*}\big) =
\sum_{i \in S} \frac{(w_i^*)^{1-p}}{\sum_{k \in S} (w_k^*)^{1-p}} \trace(M_F(\vec{w^*})^{p-1} M_i) \label{proporKKT}\\
\Longleftrightarrow 
\frac{1}{N} \sum_{k \in S} (w_k^*)^{1-p} =
\frac{\sum_{i\in S} (w_i^*)^{1-p} \trace(M_F(\vec{w^*})^{p-1} M_i)}{ \varphi_p(\vec{w^*})}. \nonumber
\end{align}
We are going to show that for all $\vec{w}\geq\vec{0}$ such that $M_F(\vec{w})$ is invertible,
$\sum_{i\in S} w_i^{1-p} \trace(M_F(\vec{w})^{p-1} M_i) \leq \trace(M_S)^p$,
where $M_S:=\sum_{i\in S} M_i$,
which will complete the proof.
To do this, we introduce the function $f$
defined on the open subset of $(\RR_+)^s$ such that $M_F(\vec{w})$ is invertible
by:
$$f(\vec{w})=\sum_{i\in S} w_i^{1-p} \trace(M_F(\vec{w})^{p-1} M_i)=\trace \left ( \Big(\sum_{i\in S} w_i^{1-p} M_i\Big) M_F(\vec{w})^{p-1} \right).$$
Note that $f$ satisfies the property $f(t\vec{w})=f(\vec{w})$ for all positive scalar $t$;
this explains why we do not have to work with normalized designs such that $\sum_i w_i = N$.
Now, let $\vec{w}\geq\vec{0}$ be such that $M_F(\vec{w})\succ 0$ and
let $k$ be an index of $S$ such that $w_k=\min_{i\in S} w_i$. We are first going to show that
$\frac{\partial f(\vec{w})}{\partial w_k}\geq0$.
By the rule of differentiation of a product,
\begin{align}
\frac{\partial f(\vec{w})}{\partial w_k} &= \trace \left( (1-p) w_k^{-p} M_k M_F(\vec{w})^{p-1} +
\Big(\sum_{i\in S} w_i^{1-p} M_i\Big) \frac{\partial (M_F(\vec{w})^{p-1})}{\partial w_k} \right) \nonumber\\
 &=\trace \left( (1-p) w_k^{-p} M_k M_F(\vec{w})^{p-1} +
\Big(\sum_{i\in S} w_i^{1-p} M_i\Big) D[x\mapsto x^{p-1}](M_F(\vec{w}))(M_k) \right) \nonumber\\
 &=\trace M_k \left(
(1-p) w_k^{-p} M_F(\vec{w})^{p-1} +  D[x\mapsto x^{p-1}]\big(M_F(\vec{w})\big)\big(\sum_{i\in S} w_i^{1-p} M_i)\big)
\right), \label{Mkparent}
\end{align}
where the first equality is simply a rewriting of $\frac{\partial (M_F(\vec{w})^{p-1})}{\partial w_k}$
by using a directional derivative, and the second equality follows from Lemma~\ref{lemma:permDeriv} applied to the
function $x\mapsto x^{p-1}$.
By linearity of the Fr\'echet derivative, we have:
$$w_k^p\  D[x\mapsto x^{p-1}]\big(M_F(\vec{w})\big)\big(\sum_{i\in S} w_i^{1-p} M_i\big)=
 D[x\mapsto x^{p-1}]\big(M_F(\vec{w})\big)\big(\sum_{i\in S} w_i \left(\frac{w_k}{w_i}\right)^{\!\!p} M_i\big)
.$$
Since $w_k\leq w_i$ for all $i\in S$, the following matrix inequality holds:
$$\sum_{i\in S} w_i \left(\frac{w_k}{w_i}\right)^{\!\!p} M_i\preceq \sum_{i\in S} w_i M_i \preceq M_F(\vec{w}).$$
By applying successively Lemma~\ref{lemma:nonincDeriv} ($x\mapsto x^{p-1}$ is antitone on $\Rplusstar$)
and Lemma~\ref{lemma:commDeriv} (the matrix $M_F(\vec{w})$ commutes with itself), we obtain:
\begin{align*}
 w_k^p\  D[x\mapsto x^{p-1}]\big(M_F(\vec{w})\big)\big(\sum_{i\in S} w_i^{1-p} M_i\big)
&\succeq  D[x\mapsto x^{p-1}]\big(M_F(\vec{w})\big)\big(M_F(\vec{w})\big) \\
& = (p-1) M_F(\vec{w})^{p-2} M_F(\vec{w})\\
& = (p-1) M_F(\vec{w})^{p-1}.
\end{align*}
Dividing the previous matrix inequality by $w_k^p$, we find that the matrix that is
inside the largest parenthesis of Equation~\eqref{Mkparent} is positive semidefinite,
from which we can conclude: $\frac{\partial f(\vec{w})}{\partial w_k}\geq0$.

Thanks to this property, we next show that
$f(\vec{w})\leq f(\vec{v})$, where $\vec{v}\in \mathbb{R}^s$ is defined by $v_i=\max_{k\in S} (w_k)$
if $i \in S$ and $v_i=w_i$ otherwise. Assume without loss of generality (after a reordering
of the coordinates) that $S=[s_0]$, $w_1\leq w_2 \leq \ldots \leq w_{s_0}$, and
denote the vector of the remaining components of $\vec{w}$ by $\vec{\bar{w}}$
(i.e., we have $\vec{w}^T=[w_1,\ldots,w_{s_0},\vec{\bar{w}}]$ and $\vec{v}^T=[w_{s_0},\ldots,w_{s_0},\vec{\bar{w}}]$).
The following inequalities hold:
$$f(\vec{w})
= f\left(\left[
\begin{array}{c}
w_1 \\w_2\\w_3\\ \vdots\\ w_{s_0} \\ \vec{\bar{w}}
\end{array}
\right]\right)
\leq f\left(\left[
\begin{array}{c}
w_2 \\w_2\\w_3\\ \vdots\\ w_{s_0} \\ \vec{\bar{w}}
\end{array}
\right]\right)
\leq f\left(\left[
\begin{array}{c}
w_3 \\w_3\\w_3\\ \vdots\\ w_{s_0} \\ \vec{\bar{w}}
\end{array}
\right]\right)
\leq \ldots \leq f\left(\left[
\begin{array}{c}
w_{s_0} \\w_{s_0}\\w_{s_0}\\ \vdots\\ w_{s_0} \\ \vec{\bar{w}}
\end{array}
\right]\right)
=f(\vec{v}).$$
The first inequality holds because $\frac{\partial f(\vec{w})}{\partial w_1}\geq0$
as long as $w_1\leq w_2$. To see that the second inequality holds, we apply the
same reasoning on the function $\tilde{f}: [w_2,w_3,\ldots] \mapsto f([w_2,w_2,w_3,\ldots])$,
i.e., we consider a variant of the problem where the matrices $M_1$ and $M_2$ have been replaced by a single matrix $M_1+M_2$.
The following inequalities are obtained in a similar manner.

Recall that we have set $M_S=\sum_{i \in S} M_i$. We have:
$$M_F(\vec{v}) = w_{s_0} M_S + \sum_{i \notin S} w_i M_i \succeq w_{s_0} M_S$$
and by isotonicity of the mapping $x \mapsto x^{1-p}$,
$M_F(\vec{v})^{1-p} \succeq (w_{s_0}\ M_S)^{1-p}$.

We denote by $X^\dagger$ the Moore-Penrose inverse of $X$.
 It is known~\cite{PS83} that
if $M_i\in\mathbb{S}_m^+$, the function $X\mapsto \trace(X^\dagger M_i)$ is nondecreasing
with respect to the L\"owner ordering over the set of matrices $X$ whose range contains
$M_i$. Hence, since $M_F(\vec{v})\succeq M_F(\vec{w})$ is invertible,
$$\forall i\in S, \quad \trace(M_F(\vec{v})^{p-1} M_i)=\trace \left(
\big( M_F(\vec{v})^{1-p} \big)^\dagger
M_i \right) \leq \trace \left( \big((w_{s_0}\ M_S)^{1-p}\big)^\dagger M_i\right)$$
and
\begin{align*}
f(\vec{v}) & = w_{s_0}^{1-p} \sum_{i\in S}  \trace(M_F(\vec{v})^{p-1} M_i) \\
&\leq w_{s_0}^{1-p} \sum_{i\in S} \trace \left( \big((w_{s_0}\ M_S)^{1-p}\big)^\dagger M_i\right) \\
& =\trace \big(M_S^{1-p}\big)^\dagger M_S\\
& = \trace M_S^p
\end{align*}
Finally, we have $f(\vec{w})\leq f(\vec{v}) \leq \trace M_S^p = \varphi_p(S),$
and the proof is complete.
\end{proof}


\begin{thebibliography}{CCPV07}

\bibitem[Atw73]{Atw73}
C.L. Atwood.
\newblock Sequences converging to {D}-optimal designs of experiments.
\newblock {\em Annals of statistics}, 1(2):342--352, 1973.

\bibitem[AZ99]{AZ99}
T.~Ando and X.~Zhan.
\newblock Norm inequalities related to operator monotone functions.
\newblock {\em Mathematische Annalen}, 315:771--780, 1999.

\bibitem[BGS08]{BGSagnol08Rio}
M.~Bouhtou, S.~Gaubert, and G.~Sagnol.
\newblock Optimization of network traffic measurement: a semidefinite
  programming approach.
\newblock In {\em Proceedings of the International Conference on Engineering
  Optimization (ENGOPT)}, Rio De Janeiro, Brazil, 2008.
\newblock {ISBN} 978-85-7650-152-7.

\bibitem[BGS10]{BGSagnol10ENDM}
M.~Bouhtou, S.~Gaubert, and G.~Sagnol.
\newblock Submodularity and randomized rounding techniques for optimal
  experimental design.
\newblock {\em Electronic Notes in Discrete Mathematics}, 36:679 -- 686, March
  2010.
\newblock ISCO 2010 - International Symposium on Combinatorial Optimization.
  Hammamet, Tunisia.

\bibitem[Bha97]{Bha97}
R.~Bhatia.
\newblock {\em Matrix analysis}.
\newblock Springer Verlag, 1997.

\bibitem[CC84]{CC84}
M.~Conforti and G.~Cornuéjols.
\newblock Submodular set functions, matroids and the greedy algorithm: tight
  worst-case bounds and some generalizations of the {R}ado-{E}dmonds theorem.
\newblock {\em Discrete applied mathematics}, 7(3):251--274, 1984.

\bibitem[CCPV07]{CCPV07}
G.~Calinescu, C.~Chekuri, M.~{P\'al}, and J.~{Vondr\'ak}.
\newblock Maximizing a submodular set function subject to a matroid constraint.
\newblock In {\em Proceedings of the 12th international conference on Integer
  Programming and Combinatorial Optimization, IPCO}, volume 4513, pages
  182--196, 2007.

\bibitem[DPZ08]{DPZ08}
H.~Dette, A.~Pepelyshev, and A.~Zhigljavsky.
\newblock Improving updating rules in multiplicative algorithms for computing
  {D}-optimal designs.
\newblock {\em Computational Statistics \& Data Analysis}, 53(2):312 -- 320,
  2008.

\bibitem[Fed72]{Fed72}
V.V. Fedorov.
\newblock {\em Theory of optimal experiments}.
\newblock New York : Academic Press, 1972.
\newblock Translated and edited by W. J. Studden and E. M. Klimko.

\bibitem[Fei98]{Fei98}
U.~Feige.
\newblock A threshold of $\operatorname{ln} n$ for approximating set cover.
\newblock {\em Journal of ACM}, 45(4):634--652, July 1998.

\bibitem[HP95]{HP95}
F.~Hansen and G.K. Pedersen.
\newblock Perturbation formulas for traces on {C*}-algebras.
\newblock {\em Publications of the research institute for mathematical
  sciences, Kyoto University}, 31:169--178, 1995.

\bibitem[HT09]{HT09}
R.~Harman and M.~Trnovsk\'a.
\newblock Approximate {D}-optimal designs of experiments on the convex hull of
  a finite set of information matrices.
\newblock {\em Mathematica Slovaca}, 59(5):693--704, December 2009.

\bibitem[IK88]{IK88}
T.~Ibaraki and N.~Katoh.
\newblock {\em Resource allocation problems: algorithmic approaches}.
\newblock MIT Press, 1988.

\bibitem[JB06]{JB06}
E.~Jorswieck and H.~Boche.
\newblock {\em Majorization and matrix-monotone functions in wireless
  communications}.
\newblock Now Publishers Inc., 2006.

\bibitem[Kie74]{Kief74}
J.~Kiefer.
\newblock General equivalence theory for optimum designs (approximate theory).
\newblock {\em The annals of Statistics}, 2(5):849--879, 1974.

\bibitem[Kie75]{Kief75}
J.~Kiefer.
\newblock Optimal design: Variation in structure and performance under change
  of criterion.
\newblock {\em Biometrika}, 62(2):277--288, 1975.

\bibitem[Kos06]{Kos06}
T.~Kosem.
\newblock inequalities between $\vert f(a+b) \vert$ and $\vert f(a)+f(b)
  \vert$.
\newblock {\em Linear Algebra and its Applications}, 418:153--160, 2006.

\bibitem[KST09]{KST09}
A.~Kulik, H.~Shachnai, and T.~Tamir.
\newblock Maximizing submodular set functions subject to multiple linear
  constraints.
\newblock In {\em SODA '09: Proceedings of the Nineteenth Annual ACM -SIAM
  Symposium on Discrete Algorithms}, pages 545--554, Philadelphia, PA, USA,
  2009.

\bibitem[L{\"o}w34]{Low34}
K.~L{\"o}wner.
\newblock {\"U}ber monotone {M}atrixfunktionen.
\newblock {\em Mathematische Zeitschrift}, 38(1):177--216, 1934.

\bibitem[Min78]{Min78}
Michel Minoux.
\newblock Accelerated greedy algorithms for maximizing submodular set
  functions.
\newblock In J.~Stoer, editor, {\em Optimization Techniques}, volume~7 of {\em
  Lecture Notes in Control and Information Sciences}, pages 234--243. Springer
  Berlin / Heidelberg, 1978.
\newblock 10.1007/BFb0006528.

\bibitem[MM76]{MM76}
T.L. Morin and R.E. Marsten.
\newblock An algorithm for nonlinear knapsack problems.
\newblock {\em Management Science}, pages 1147--1158, 1976.

\bibitem[MO79]{MO79}
AW~Marshall and I.~Olkin.
\newblock {\em Inequalities: Theory of Majorization and its Applications}.
\newblock Academic Press, 1979.

\bibitem[NWF78]{NWF78}
G.L. Nemhauser, L.A. Wolsey, and M.L. Fisher.
\newblock An analysis of approximations for maximizing submodular set
  functions.
\newblock {\em Mathematical Programming}, 14:265--294, 1978.

\bibitem[PR92]{PR92}
F.~Pukelsheim and S.~Rieder.
\newblock Efficient rounding of approximate designs.
\newblock {\em Biometrika}, pages 763--770, 1992.

\bibitem[PS83]{PS83}
F.~Pukelsheim and G.P.H. Styan.
\newblock Convexity and monotonicity properties of dispersion matrices of
  estimators in linear models.
\newblock {\em Scandinavian journal of statistics}, 10(2):145--149, 1983.

\bibitem[Puk80]{Puk80}
F.~Pukelsheim.
\newblock On linear regression designs which maximize information.
\newblock {\em Journal of statistical planning and inferrence}, 4:339--364,
  1980.

\bibitem[Puk93]{Puk93}
F.~Pukelsheim.
\newblock {\em Optimal Design of Experiments}.
\newblock Wiley, 1993.

\bibitem[RS89]{RS89}
TG~Robertazzi and SC~Schwartz.
\newblock {An Accelerated Sequential Algorithm for Producing $ D $-Optimal
  Designs}.
\newblock {\em SIAM Journal on Scientific and Statistical Computing}, 10:341,
  1989.

\bibitem[Sag11]{Sagnol09SOCP}
G.~Sagnol.
\newblock Computing optimal designs of multiresponse experiments reduces to
  second-order cone programming.
\newblock {\em Journal of Statistical Planning and Inference}, 141(5):1684 --
  1708, 2011.

\bibitem[SGB10]{SagnolGB10ITC}
G.~Sagnol, S.~Gaubert, and M.~Bouhtou.
\newblock Optimal monitoring on large networks by successive c-optimal designs.
\newblock In {\em 22nd international teletraffic congress (ITC22), Amsterdam,
  The Netherlands}, September 2010.

\bibitem[SQZ06]{SQZ06}
H.H. Song, L.~Qiu, and Y.~Zhang.
\newblock Netquest: A flexible framework for largescale network measurement.
\newblock In {\em ACM SIGMETRICS'06}, St Malo, France, 2006.

\bibitem[Svi04]{Svi04}
M.~Sviridenko.
\newblock A note on maximizing a submodular set function subject to a knapsack
  constraint.
\newblock {\em Operation Research Letters}, 32(1):41--43, 2004.

\bibitem[UP07]{UP07}
D.~Uci\'{n}ski and M.~Patan.
\newblock D-optimal design of a monitoring network for parameter estimation of
  distributed systems.
\newblock {\em Journal of Global Optimization}, 39(2):291--322, 2007.

\bibitem[Von08]{Von08}
J.~Vondr\'ak.
\newblock Optimal approximation for the submodular welfare problem in the value
  oracle model.
\newblock In {\em ACM Symposium on Theory of Computing, STOC'08}, pages 67--74,
  2008.

\bibitem[Wol82]{Wol82}
L.A. Wolsey.
\newblock Maximising real-valued submodular functions: Primal and dual
  heuristics for location problems.
\newblock {\em Mathematics of Operations Research}, pages 410--425, 1982.

\bibitem[Yu10]{Yu10a}
Y.~Yu.
\newblock Monotonic convergence of a general algorithm for computing optimal
  designs.
\newblock {\em The Annals of Statistics}, 38(3):1593--1606, 2010.

\bibitem[Zha02]{Zhan02}
X.~Zhan.
\newblock {\em Matrix Inequalities (Lecture Notes in Mathematics)}.
\newblock Springer, 2002.

\end{thebibliography}
\end{document}